\documentclass[12pt]{amsart}
\usepackage[colorlinks=true,citecolor=blue,linkcolor=blue]{hyperref}
\usepackage{latexsym,amsmath,amssymb}
\usepackage{accents}
  \usepackage{cleveref}
\usepackage{bm}
\usepackage{mathrsfs}
\usepackage{tikz} 

\usepackage[dvips,bottom=1.4in,right=1in,top=1in, left=1in]{geometry}

\def\eps{\varepsilon}

\def\N{{\mathbb N}}

\def\S{{\mathbb S}}

\newcommand{\sm}{\setminus}

\newtheorem{theorem}{Theorem}
\newtheorem{lemma}[theorem]{Lemma}
\newtheorem{corollary}[theorem]{Corollary}
\newtheorem{proposition}[theorem]{Proposition}
\theoremstyle{definition}
\newtheorem{remark}[theorem]{Remark}
\newtheorem{definition}[theorem]{Definition}

\def\esssup{\mathop{\rm ess\,sup\,}}

\def\dist{{\rm dist\,}}

\def\lip{{\rm Lip\,}}

\def\supp{{\rm supp\,}}

\newcommand{\R}{\mathbb{R}}

\newcommand{\brac}[1]{\left (#1 \right )}
\newcommand{\abs}[1]{\left |#1 \right |}
\newcommand{\Ep}{\bigwedge\nolimits}

\newcommand{\aleq}{\precsim}

\newcommand{\laps}[1]{|D|^{#1}}

\newcommand{\lapms}[1]{I^{#1}}

\newcommand{\dG}{\operatorname{d}_\alpha\!}  

\newcommand{\Div}{\operatorname{Div}}
\newcommand{\divR}{\operatorname{Div}_\alpha}  
\newcommand{\divG}{\operatorname{div}_\alpha}  

\newcommand{\VR}{{{\rm Var}_\alpha\!}}   
\newcommand{\VG}{{{\rm var}_\alpha\!}}   

\newcommand{\BVG}{{bv}_\alpha} 

\newcommand{\BVRz}{{{\rm BV}_{00}^\alpha(\Omega)}}

\newcommand{\phiR}{{\Phi}}              
\newcommand{\phiG}{\Phi}                    

\DeclareMathOperator*{\argmin}{arg\,min} 
\newcommand{\BVO}{\mathrm{BV}(\Omega)}
\newcommand{\Da}{{D^\alpha}}
\renewcommand{\div}{\operatorname{div}}

\newcommand{\Cc}{C_c} 
\newcommand{\Cinf}{\mathscr{C}^\infty}

\newcommand{\Per}{\rm Per}	%
\newcommand{\Var}{{\rm Var}}  
\newcommand{\var}{{\rm var}}  
\newcommand{\Lb}[1]{\mathcal{L}^{#1}} %
\providecommand{\abs}[1]{\left| #1\right|} 
\providecommand{\norm}[1]{\left\| #1\right\|}
\providecommand{\prts}[1]{\left( #1 \right)}

\usepackage[textsize=tiny]{todonotes}
\author{Harbir Antil}
\address[Harbir Antil]{Department of Mathematical Sciences and the Center for Mathematics and Artificial Intelligence (CMAI) 
George Mason University, Fairfax, VA 22030, USA. \newline
ORCiD: 0000-0002-6641-1449}
\email{hantil@gmu.edu}

\author{Hugo D\'iaz}
\address[Hugo D\'iaz]{Department for Mathematical Sciences,
University of Delaware, Newark, DE 19716, USA.  
}
\email{hugodiaz@udel.edu}

\author{Tian Jing}
\address[Tian Jing]{Department of Mathematics, University of Pittsburgh, Pittsburgh, PA 15261, USA.
}
\email{tij11@pitt.edu}

\author{Armin Schikorra}
\address[Armin Schikorra]{Department of Mathematics, University of Pittsburgh, Pittsburgh, PA 15261, USA.\newline
ORCiD: 0000-0001-9242-1782}
\email{armin@pitt.edu}

\title{Nonlocal Bounded Variations with Applications}

\thanks{HA and HD are partially supported by NSF grants DMS-2110263, DMS-2111315, DMS-1913004, 
the Air Force Office of Scientific Research under Award NO: FA9550-22-1-0248. TJ and AS are partially supported under NSF Career DMS-2044898. AS is partially supported by Simons foundation grant no 579261}

\numberwithin{theorem}{section} \numberwithin{equation}{section}

\begin{document}

\begin{abstract}
Motivated by problems where jumps across lower dimensional subsets and sharp transitions
across interfaces are of interest, this paper studies the properties of fractional bounded variation ($BV$)-type spaces. 
Two different natural fractional analogs of classical $BV$ are considered:  $BV^\alpha$, a space induced from the Riesz-fractional gradient that has been recently studied by Comi-Stefani; and $bv^\alpha$, induced by the Gagliardo-type fractional gradient often used in Dirichlet forms and Peridynamics -- this one is naturally related to the Caffarelli-Roquejoffre-Savin fractional perimeter. Our main theoretical result is that the latter $bv^\alpha$ actually corresponds to the Gagliardo-Slobodeckij space $W^{\alpha,1}$. As an application, using the properties of these spaces, novel image denoising models are introduced and their corresponding Fenchel 
pre-dual formulations are derived. The latter requires density of smooth functions with compact support. We establish this density property for convex domains.
\end{abstract}

\keywords{
fractional derivatives, fractional BV, image denoising
}

\subjclass[2010]{
35R11,       
49Jxx,        
65R20       
65D18       
}

\maketitle

\tableofcontents
\section{Introduction}\label{sec:intro}
In recent years, fractional calculus and nonlocal operators have emerged as natural 
tools to study various phenomena in science and engineering. Unlike their classical
counterparts, fractional operators have several distinct abilities, for instance, they 
require less smoothness and they are nonlocal in nature. Such flexibilities have led 
to multiple successes of fractional derivative-based models in practical applications. 
For instance, magnetotellurics in geophysics \cite{CJWeiss_BGVBWaanders_HAntil_2020a}, 
viscoelastic models  \cite{Mainardi2010}, quantum spin chains and harmonic maps 
\cite{DLR2011,ELenzmann_ASchikorra_2018a,HAntil_SBartels_ASchikorra_2021a},
deep neural networks \cite{HAntil_RKhatri_RLohner_DVerma_2020a}, repulsive curves
\cite{Crane2021}, etc. 

A fundamental concept in inverse problems, such as image denoising, is the use of 
regularization. The article \cite{HAntil_SBartels_2017a} introduced the
fractional Laplacian as a regularizer in image denoising as an alternative to well-known
approaches such as total-variation regularization. Subsequently, this model 
has been successfully used by various authors in imaging science as it provides a 
behavior that is closer to total variation based approaches \cite{JAIglesias_MMercier_2022a}, 
but it is easy to implement in practice. The current paper is motivated by these observations. 
We also refer to \cite{Osher2007} for a different (discrete) nonlocal regularization in imaging.

Fundamental developments are being made in fractional calculus. In fact, now 
there exist notions of fractional divergence and gradient. The aforementioned fractional 
Laplacian, for instance, can be obtained by the composition of fractional divergence and 
fractional gradient. This is similar to the classical integer order setting. Such  
discoveries are not only fueling further developments in analysis  but are also leading to 
new application areas or improving the existing ones. Motivated by image denoising, 
the goal of this paper is to 
study fundamental properties of the space of (nonlocal) fractional bounded variation. Based 
on such fractional order spaces, we introduce novel image denoising models, and we derive
Fenchel dual formulations \cite[chapter III]{Ekeland99} for these. 
Notice that such formulations are critical in deriving efficient numerical methods in 
the classical setting.
The remainder of this section provides a precise discussion on new image denoising 
models to motivate the analytical tools developed in this paper. 

A well-established  method to solve image denoising problems is based on {\it total variation minimization}  
\cite{Alvarez1992,RudinThesis,ROF}. 
Let ${ u_N:\Omega \subseteq \R^n \rightarrow \R}$  denote a continuous representation 
of an image (possibly noisy). Given a regularization parameter $\beta > 0$, a standard 
image denoising problem amounts to finding $u$ solving 
\begin{align}
	\argmin_{u \in \mathscr{X}} \left\{ \beta |Du|_{ \mathscr{X}}
		+ \frac{1}{2} \| u-u_N \|_{L^2(\Omega)}^2 \right\}, \label{BVproblem}
\end{align}
where the space $ \mathscr{X}$ is chosen in conjunction  with the norm $|\cdot|_{ \mathscr{X}}$ such that 
$D u$ is well defined at least in a distributional sense, and  $u$ can be  piecewise smooth. 
In  practice, one of the most common spaces used  is the space of functions with  {\it bounded variation} (BV) defined by 
\begin{align*}
\BVO= \left\{ u\in L^1(\Omega)\!: \Var(u;\Omega) < \infty\right\}.
\end{align*}
Namely, 
a function $u$ in  $ L^1(\Omega,\R)$ is said to 
have bounded variation if and only if
\[
\Var(u;\Omega) := \sup \left \{ \int_{{\R^n}} f(x)\, \Div \Phi(x)\, dx \, : \,  \Phi \in \Cc^1(\Omega,\R^n), \|\Phi\|_{L^\infty(\Omega)} \leq 1  \right \}  < \infty.
\]
If the variation $\Var(u;\Omega)$ is finite, one can show that its distributional derivative $Du$ is a Radon measure and $\Var(u;\Omega) = |Du|(\Omega)$, see \cite[Ch. 10]{Attouch2014}. 
It is  well-known that $\BVO$  preserves edges, in a noisy image, better than $W^{1,1}(\Omega)$
while retaining several of its properties. For instance, it is a Banach space, it is lower semi-continuous 
on $L^1(\Omega)$, Sobolev inequalities,  etc. 

In this work, we are interested in the fractional version of the problem \eqref{BVproblem}.
For this, we first need to decide on a notion of fractional $BV$. We do so by replacing in the definition above the derivative with some suitable fractional derivative. Alas, there are many different, yet natural, fractional operators that are considered extensions of the usual gradient -- and each one induces its own $BV$-space. 

We will consider the two most popular notions. 
Firstly, we will consider the space $BV^\alpha$, which we refer to as \emph{Riesz-type}. The study of $BV^\alpha$ was initiated by Comi-Stefani in \cite{CS19}, see \Cref{s:fracbvriesz}. It relies on the notion of what is sometimes referred to as  Riesz gradient $D^\alpha$, which is simply the usual gradient combined with a regularizing Riesz potential.

The other type of fractional $BV$ we consider will be denoted by $bv^\alpha$ and is referred to as \emph{Gagliardo}-type, see \Cref{s:bvalphagag}. We are not aware whether this has been considered in the literature prior to this work. The notion of fractional derivative is what we {will refer} to as the Gagliardo-type derivative considered in various  aspects of mathematics, e.g. Dirichlet forms \cite{H15}, Peridynamics \cite{DGLZ13} and harmonic analysis \cite{MS18}. This Gagliardo-type $bv^\alpha$ is naturally related to the most popular notion of a fractional perimeter  defined by Caffarelli--Roquejoffre--Savin \cite{MCRS10}. Indeed, we will show in \Cref{th:bvsws1} that $bv^\alpha$ coincides with the Gagliardo-Sobolev space $W^{\alpha,1}$ -- a maybe surprising  feature of the case $\alpha < 1$, since this is false for $\alpha = 1$: indeed it is well-known that $W^{1,1} \neq BV$, see \cite{EG15}. This is one of the main theoretical contributions of the current paper.

Based on these notions of fractional $BV$, we will introduce new types of variational models for image denoising. Namely, we study the fractional versions of \eqref{BVproblem},
\begin{align}
	\argmin_{u \in \mathscr{X}} \left\{ \beta\Var_\alpha(u;\Omega)
		+ \frac{\gamma}{p} \| u-u_N\|_{{L^p(\Omega)}}^p \right\} . \label{BVproblem_non}
\end{align}
A related model was studied by Bartels and one of the authors in \cite{HAntil_SBartels_2017a} but working in fractional order Hilbert space $H^s(\Omega)$ instead of ${\mathscr{X}} = BV^\alpha(\Omega)$.

We emphasize that the numerical algorithms for solving problems of type 
\eqref{BVproblem} make use of the Fenchel dual formulations 
\cite{SBartels_2015a,AChambolle_TPock_2011a}. However, this requires dealing with 
the dual space of $\BVO$, whose full characterization is still an unknown \cite{Torres}. 
Instead one proceeds by finding a predual problem to \eqref{BVproblem}, i.e., a 
problem whose Fenchel conjugate is \eqref{BVproblem}, see for instance 
\cite{Burger,AChambolle_2004a, Hintermuller}. In this case, one does not need to 
deal with $\BVO^*$, but instead the closure in $L^p(\Omega)$ of the range of a  
 divergence-like operator,  which is the conjugate of  
 $ -D:\mathscr{X}\subset \BVO  \rightarrow \mathcal{M}(\Omega,\R^n)$.
 We will derive a pre-dual problem 
corresponding to \eqref{BVproblem_non} in \Cref{s:imagedenoise}. 
Derivation of pre-dual requires density of smooth functions with
compact support. This is highly non-trivial in general even in the local case.
We establish this result provided that the domain $\Omega$ is convex. Such results are of interest independent of this 
paper, see Propositions~\ref{pr:IbetaeqItildebetaRiesz} and \ref{pr:IbetaeqItildebetaGag}.

\section{Fractional BV in the Riesz sense}\label{s:fracbvriesz}

We begin by recalling the notion of fractional Laplacian and its inverse, the Riesz potential.
 Denote by $\mathcal{F}$ and $\mathcal{F}^{-1}$ the Fourier transform on $\R^n$. For $\alpha > 0$ the fractional Laplacian of $f: \R^n \to \R$ with differential order $\alpha$, denoted by $\laps{\alpha}f$, 
 is given by 
\[
 \laps{\alpha} f(x) := \mathcal{F}^{-1} \brac{|\xi|^{\alpha} \mathcal{F} f(\xi)}(x).
\]
The notation $\laps{\alpha} = (-\Delta)^{\frac{\alpha}{2}}$ is common, 
but we will mostly use the notation $|D|^\alpha$ in this paper, 
since it  states the order of derivatives more clearly. 
The definition above makes sense when $\alpha < 0$.
 In that case, we call the operator Riesz potential. 
 More precisely, for all  $\alpha \in (0,n)$ we define 
\[
 \lapms{\alpha} f(x) := \mathcal{F}^{-1} \brac{|\xi|^{-\alpha} \mathcal{F} f(\xi)}(x).
\]
It is then easy to see that $\laps{\alpha}\lapms{\alpha} f=\lapms{\alpha}\laps{\alpha} f= f$,
 at least for suitably smooth functions {with decay} at infinity, 
i.e. the fractional Laplacian and Riesz potential are inverses to each other.
The fractional Laplacian $\laps{\alpha}$ has no gradient structure. 
It does not converge to the gradient $D$ when $\alpha \to 1$. 
Recently, many authors considered a fractional-order  operator with a gradient structure. 
Although this operator {can be traced} as far back as \cite{H59}, it has received increased interest in various applications since the works e.g. \cite{CS19,S15,S15b,SS15}. It is defined very simply as the usual gradient of the Riesz potential
\begin{align}\label{def:DaR}
 \Da f:= D \lapms{1-\alpha} f.
\end{align}
From its  Fourier transform representation, it is easy to show that $\Da \to D$ as $\alpha \to 1$. 
The fractional divergence $\divR$ is defined as 
\[
 \divR f = \div \lapms{1-\alpha} f.
\]
Note that  $\divR$ is the  adjoint of $-\Da$. In fact, the following integration-by-parts formula holds 
\begin{align} \label{eq:IBPR}
 \int_{\R^n} F \cdot \Da g dx= -\int_{\R^n} \divR F\, gdx  \qquad \forall  F \in  \Cc^\infty(\R^n,\R^n),\forall g \in \Cc^\infty(\R^n),
\end{align}
which follows readily from the definition via the Fourier transform and Plancherel's theorem. 
We remark on the integral definition of the above operators.
For any $\alpha \in (0,1]$, we have
\begin{equation}\label{eq:integralrep}
\begin{split}
\laps{\alpha} f(x) &= c_{1,\alpha} \int_{\R^n} \frac{f(x)-f(y)}{|x-y|^{n+\alpha}} dy,\\
 D^\alpha  f(x) &= c_{2,\alpha} \int_{\R^n} \frac{\brac{f(x)-f(y)} (x-y)}{|x-y|^{n+\alpha+1}} dy,\\
 \divR F(x) &= c_{3,\alpha} \int_{\R^n} \frac{\brac{F(x)-F(y)}\cdot (x-y)}{|x-y|^{n+\alpha+1}} dy,
\end{split}
 \end{equation}
for some constants $c_{1,\alpha}$,  
$c_{2,\alpha}$ and $c_{3,\alpha}$, 
which can be found in the literature.
Having the notion of fractional gradient, we {naturally obtain an} associated notion of fractional BV spaces.
 Our definitions are very similar to \cite{CS19} and different from other natural approaches  
 as in \cite{BGJ14} or an approach via a different type of nonlocal gradient and divergence as in \cite{DGLZ13,MS18}, which we will discuss later in \Cref{s:bvalphagag}.
When using the {notion of fractional BV spaces in this paper, 
most of the required properties will follow 
similar general principles as} in typical BV spaces. We provide a derivation
of the results that we could not find in the literature and provide references
otherwise. It is possible that some of these results are known to the experts.

To distinguish the resulting space from the one discussed in \Cref{s:bvalphagag}, we use the notations $BV^\alpha$, {${\rm Div}_\alpha$} and $\Var_{\alpha}$.
In \Cref{s:bvalphagag} we will use $bv^\alpha$, $\div_\alpha$, and $\var_\alpha$ instead.

Let $\alpha \in (0,1]$ and  $f \in L^1({\R^n})$,
the variation of $f$ is defined as  

\begin{align} \label{Def:Valpha}
\Var_\alpha(f;\R^n) := \sup \left \{ \int_{{\R^n}} f\, \Div_\alpha \Phi\, dx \, : \,  {\Phi} \in \Cc^1(\R^n;\R^n), \|\Phi\|_{L^\infty(\R^n)} \leq 1  \right \} . 
\end{align}
Let $\Omega\subseteq \mathbb{R}^n$. For any $f \in L^1(\Omega)$, we define
\[
\Var_{\alpha}(f;\Omega) := \Var_\alpha(\chi_{\Omega} f;\R^n),
\]
where $\chi_{\Omega} f$ is the extension of $ f$ by zero to $\R^n$.
The integral
\[
\int_{{\R^n}} f\, \Div_\alpha \Phi\, dx 
\] 
is well defined for all $f \in L^1(\R^n)$ and $\Phi \in \Cc^1(\R^n {, \R^n})$, which is a consequence of the following result.

\begin{lemma}\label{la:divalphaphi}
Let $\Phi \in \Cc^1(\R^n;\R^n)$,
 then for any $\alpha \in (0,1]$ and any  $p \in [1,\infty]$ we have 
\[
 \Div_\alpha \Phi \in L^p(\R^n).
\]
\end{lemma}
\begin{proof}
Fix $\Phi \in \Cc^1(\R^n;\R^n)$. 
For $\alpha = 1$, 
we have $\Div_\alpha \Phi \in \Cc (\R^n)\subseteq L^p(\R^n)$ for all $p \in [1,\infty]$.
For $\alpha < 1$, we have from \eqref{eq:integralrep} that
\[
\abs{\Div_\alpha \Phi (x)} \aleq_{\alpha} \brac{2\| \Phi  \|_{L^\infty(\R^n)}  +  \|\nabla \Phi \|_{L^\infty(\R^n)}
 }
  \, \int_{\R^n} \frac{\min\{1,|x-y|\} }{|x-y|^{n+\alpha}} dy.
\]
Here $\aleq_{\alpha}$ implies that the hidden constant depends on $\alpha$ (and any constant may depend on the dimension {$n$}). 
Since $\alpha < 1$, the following integral is finite 
and {has} the same value for every $x \in \R^n$, 
i.e.
\[
 \int_{\R^n} \frac{\min\{1,|x-y|\} }{|x-y|^{n+\alpha}} dy \equiv C(n,\alpha) <\infty,
\]
which implies that
\[
 \|\Div_\alpha \Phi \|_{L^\infty(\R^n)} \aleq_{\alpha} \brac{\| \Phi \|_{L^\infty(\R^n)}  +  \|\nabla  \Phi \|_{L^\infty(\R^n)}}.
\]

It remains to prove that $\Div_\alpha \Phi \in L^1(\R^n)$. 
Once this is shown we conclude $\Div_\alpha \Phi \in L^p(\R^n)$ for any $p \in [1,\infty]$ by interpolation.
Taking $R \geq 1$ large enough, such that $\supp \Phi \subset B(0,R/2)$,
then for $x \in \R^n \sm B(0,R)$ we have 
\[
 \abs{\Div_\alpha \Phi(x)} \aleq_{\alpha} \int_{B(0,R/2)} \frac{\abs{\Phi(y)}}{|x-y|^{n+\alpha}} dy.
\]
By Fubini's theorem, we have  
\[
 \begin{split}
  \|\Div_\alpha \Phi \|_{L^1(\R^n\sm B(0,R))} \aleq& \int_{B(0,R/2)} |\Phi (y)| \brac{\int_{\R^n\sm B(0,R)} \frac{1}{|x-y|^{n+\alpha}} dx}\, dy\\
  \aleq& \| \Phi \|_{L^1(\R^n)} \sup_{y \in \R^n \sm B(0,R/2)} \brac{\int_{ \{ x: |x-y| \geq R/2 \} } \frac{1}{|x-y|^{n+\alpha}} dx}.
  \end{split}
 \]
Here we hide the constant by using $\aleq$. Using the fact that
\[
 \int_{\{ x: |x-y| \geq R/2\} } \frac{1}{|x-y|^{n+\alpha}} dx \aleq_{\alpha} R^{-\alpha} < \infty,
\]
we obtain 
\[
 \|\Div_\alpha \Phi \|_{L^1(\R^n\sm B(0,R))} \aleq_{\alpha} \| \Phi \|_{L^1(\R^n)}.
\]
On the complement $B(0,R)$, 
we have $\Div_\alpha \Phi \in L^\infty(B(0,R)) \subset L^1(B(0,R))$.
Thus, we obtain that $\|\Div_\alpha \Phi \|_{L^1(\R^n)}<\infty$,
which finishes the proof.
\end{proof}

Now we are ready to define the first fractional BV space of this work, i.e.  $BV^\alpha$, see also \cite{CS19,CS19b,Comi_2022} where this space was considered first. This space inherits most of its properties from the gradient structure of the Riesz-derivative $D^\alpha$, cf. (\ref{def:DaR}). 

\begin{definition}[Riesz-type fractional BV]
For $\Omega \subset \R^n$, we define
\begin{align}
  BV_{00}^\alpha(\Omega) := \{f \in L^1(\R^n): \quad f \equiv 0 \quad on \quad \R^n \setminus \Omega, \quad \Var_{\alpha}(f;\Omega) < \infty\},
  \label{def:BV00}
\end{align}
endowed with the norm
\[
 \|f\|_{BV^\alpha(\Omega)} := \|f\|_{L^1(\Omega)} + \Var_{\alpha}(f;\Omega).
\]
\end{definition}

In this paper, we  often identify $f \in L^1(\Omega)$ with its extension by zero $\chi_{\Omega} f \in L^1(\R^n)$. 
Observe that we do not need to assume any regularity of $\partial \Omega$ in the above (and below) definitions and results. 
The regularity of $\partial \Omega$ only comes into play when 
we consider whether constant functions in $\Omega$ belong to $BV^\alpha(\Omega)$. 
Namely $1 \in L^1(\Omega)$ belongs to $BV_{00}^\alpha(\Omega)$ (with the usual identification $1 \in L^1(\Omega)$ corresponds to $\chi_{\Omega} \in L^1(\R^n)$) if the $\alpha$-Cacciopoli-perimeter of $\partial \Omega$ is finite. We refer to \cite{CS19} for the definition of this perimeter. Essentially by definition we immediately obtain
\begin{proposition}\label{pr:Jalphaconstfinite} 
The surface $\partial \Omega$ has finite $\alpha$-Cacciopoli-perimeter, 
i.e. $\Per _\alpha (\partial \Omega )<\infty$ if and only if $\Var_{\alpha}(1;\Omega) < \infty$.
\end{proposition}
Observe that the Cacciopoli-perimeter above is different from the more commonly used fractional perimeter introduced by Caffarelli-Roquejoffre-Savin \cite{MCRS10}. The latter one is related to the fractional version of BV functions defined using the divergence as used in, e.g. \cite{DGLZ13,MS18}. We shall discuss it in \Cref{s:bvalphagag}.

Next, we note that one can obtain the existence of the distributional derivative $D^\alpha f$ (which is a Radon measure) just like for $BV$, see \cite[p.167, Theorem 1, Structure Theorem]{EG15}
If $f \in BV_{00}^\alpha(\Omega)$, then the mapping
\[
 \Cc^1(\R^n{;\R^n}) \ni \Phi \mapsto \int_{\R^n} f\Div_\alpha \Phi dx
\]
extends to a linear functional on $(C_{c}(\R^n;\R^n),\|\cdot\|_{L^\infty(\R^n)})$. By the Riesz representation theorem \cite[Section~1.8, Theorem~1]{EG15}, there exists a Radon measure $\mu$ on $\R^n$ and a $\mu$-measurable function $\sigma: \R^n \to \R^n$ such that $|\sigma|  = 1$ $\mu$-a.e. and 
\[
 \int_{\R^n} f\Div_\alpha \Phi  dx = -\int_{\R^n} \Phi \cdot \sigma d\mu.
\]
Moreover, we have 
\[
 |\mu(\R^n)| \leq \Var_{\alpha}(f;\Omega).
\]
The latter follows by the definition of the norm.
By slight abuse of notation we will denote by $D^\alpha f$ both the distributional derivative and the measure $D^\alpha f := \sigma \llcorner \mu$ (where $\llcorner$ denotes the concatenation of function and measure), whichever is applicable.

We now consider the approximation of {$BV^\alpha_{00}(\Omega)$} functions by smooth functions.
Since $f$ is compactly supported, the convolution $f \ast \eta_\eps$ is in $\Cc^\infty(\R^n)$. 
Using the same argument as in \cite[Theorem 5.2]{EG15}, we obtain the following result.
\begin{proposition}\label{pr:smoothapprox}
Let $\Omega \subset \R^n$ be open and bounded. For any $f \in BV^\alpha_{00}(\Omega)$ there exists $f_k \in \Cc^\infty(\R^n)$ such that 

\[
 \|f_k-f\|_{L^1(\R^n)} + \abs{\Var_\alpha(f;\R^n)-\Var_\alpha(f_k;\R^n)} \xrightarrow{k \to \infty} 0.
\] 
Equivalently, (since $f$ vanishes outside of $\Omega$), 
\[
 \|f_k-f\|_{L^1(\R^n)} + \abs{\Var_\alpha(f;\Omega)-\Var_\alpha(f_k;\R^n)} \xrightarrow{k \to \infty} 0.\]
\end{proposition}

We also have the following embedding theorem.
\begin{proposition} \label{Prop:embeddingLp}
Let $\Omega \subset \R^n$ be open and bounded and $n \geq 2$. 
Then for all $p \in [1,\frac{n}{n-\alpha}]$
we have $ BV^\alpha_{00}(\Omega)\subseteq L^p(\R^n)$ 
and
\[
 \|f\|_{L^p(\R^n)} \leq C(n,p,\alpha) \|f\|_{BV_{00}^\alpha(\Omega)}.
\]
If $n = 1$, then the same results hold for all $p \in \left[1, \frac{1}{1-\alpha}\right)$.
\end{proposition}
\begin{proof}
Let $f_k$ be the approximation of $f$ as in \Cref{pr:smoothapprox}.
By the main result in \cite{SSS17}, for all $ p \in \left[1,\frac{n}{n-\alpha}\right]$ we have 
\[
 \|f_k\|_{L^p(\R^n)} \leq C \brac{\|f_k\|_{L^1(\R^n)} + \|D^\alpha f_k\|_{L^1(\R^n)}},
\]
since $f_k \in \Cc^\infty(\R^n)$. Observe that by an integration-by-parts formula, since we already know {$D^\alpha f_k \in L^1(\R^n,\R^n)$},
\[
 \|D^\alpha f_k\|_{L^1(\R^n,\R^n)} = \Var_{\alpha}(f_k;\R^n).
\]
Since up to subsequences $f_k$ converges to $f$ almost everywhere we conclude from Fatou's lemma,
\[
\begin{split}
 \|f\|_{L^p(\R^n)} \leq& \liminf_{k \to \infty}\|f_k\|_{L^p(\R^n)} \leq C \liminf_{k \to \infty} \brac{\|f_k\|_{L^1(\R^n)} + \Var_{\alpha}(f_k;\R^n)}\\
 =& C\left(\|f\|_{L^1(\R^n)} + \Var_{\alpha}(f;\R^n)\right).
 \end{split}
\]
{The proof is complete.}
\end{proof}
Using the duality definition of $\Var_\alpha$ and the same argument as in \cite[Theorem 5.2]{EG15},
we obtain the lower semicontinuity with respect to the so called intermediate convergence,
see Definition 10.1.3 and Remark 10.1.3 in \cite{Attouch2014} for details.
\begin{proposition}[Lower semicontinuity]\label{lsc}
Let $\Omega \subset \R^n$ be open and bounded. Assume ${\{f_k\}_{k =1}^{\infty} \subset BV^\alpha_{00}(\Omega)}$, and assume that $f \in L^1(\R^n)$ such that 
\[
\|f_k-f\|_{L^1(\R^n)} \xrightarrow{k \to \infty} 0.
\]
Then $f \in BV^\alpha_{00}(\Omega)$ and we have 
\[
\Var_\alpha(f;\R^n) \leq \liminf_{k \to \infty}\Var_\alpha(f_k;\R^n).
\]
Or, equivalently,
\[
\Var_\alpha(f;\Omega) \leq \liminf_{k \to \infty}\Var_\alpha(f_k;\Omega).
\]

\end{proposition}

\begin{corollary}
Let $\Omega \subset \R^n$ be bounded. Then $\left(BV^\alpha_{00}(\Omega),  \|\cdot\|_{BV^\alpha(\Omega)}\right)$ is a complete space.
\end{corollary}
\begin{proof}
Let $\{f_k\}_{k=1}^{\infty}$ be a Cauchy sequence in $BV^\alpha_{00}(\Omega)$. 
Since  $f_k$ is Cauchy in $L^1(\R^n)$,
there exists  $f \in L^1(\R^n)$ with $f \equiv 0$ in $\R^n \sm \Omega$,
 such that $f_k\to f$ in $L^1(\R^n)$. 
 By \Cref{lsc}, we find that $f \in BV^\alpha_{00}(\Omega)$. 
 Using the lower semicontinuity of the variation still from \Cref{lsc}, we obtain 
\[
 \lim_{k \to \infty}\Var_\alpha(f-f_k;\Omega) \leq \lim_{k \to \infty}\liminf_{\ell \to \infty} \Var_\alpha(f_\ell-f_k;\Omega) = 0,
\]
which completes the proof.
\end{proof}

Using the weak*-convergence of Radon measures, and the arguments of the standard Rellich-Kondrachov compactness, see \cite[Theorem 5.2 \& Theorem 5.5]{EG15},
we have the following result.
\begin{proposition}[Weak compactness] \label{weak_comp}
Let $\Omega \subset \R^n$ be open and bounded. Assume $\{f_k\}_{k=1}^{\infty} \subset BV^\alpha_{00}(\Omega)$ such that 
\[
\sup_{k \geq 1} \|f_k\|_{BV^\alpha(\Omega)} < \infty.
\]
Then there exists $f \in BV_{00}^\alpha(\Omega)$ such that
\[
 \Var_\alpha(f;\R^n) \leq \liminf_{k \to \infty} \Var_\alpha(f_k;\R^n),
\]
or equivalently,
\[
 \Var_\alpha(f;\Omega) \leq \liminf_{k \to \infty} \Var_\alpha(f_k;\Omega),
\]
and there is a subsequence $\{f_{k_i}\}_{i=1}^{\infty}$ such that 
for all $p \in \left[1,\frac{n}{n-\alpha}\right)$ we have
\[
 \|f_{k_i} - f\|_{L^p(\R^n)} \xrightarrow{i \to \infty} 0.
\]

\end{proposition}

Lastly, as in the local case where we know that $H^{1,1}(\Omega)$ is a  subspace of $BV(\Omega)$ (where $H^{1,1}(\Omega)$ is the space of functions $f \in L^1(\Omega)$ such that $Df \in L^1(\Omega)$), the corresponding result for the fractional situation holds as well.

\begin{lemma}\label{la:Halpha1}
Let $f \in  H^{\alpha,1}(\R^n)$, i.e. {$f\in L^1(\R^n)$ and $D^\alpha f\in L^1(\R^n;\R^n)$}.
Assume additionally that $f \equiv 0$ in $\R^n \sm \Omega$. Then $f \in BV_{00}^\alpha(\Omega)$.
\end{lemma}
\begin{proof}
We only need to show ${\rm Var}_{\alpha} (\chi_{\Omega} f;\R^n) <\infty$.
For any $\Phi\in \Cc^1(\mathbb{R}^n;\R^n)$ {such that  $\|\Phi\|_{L^\infty(\R^n)}\leq 1$}, we have by Fubini's theorem 
\begin{equation*}
\int_{\mathbb{R}^n} \chi_{\Omega}f\, {\rm Div}_{\alpha}\Phi 
=
-  \int_{\R^n} D^{\alpha} f\cdot  \Phi 
\leq \norm{\Phi}_{L^\infty(\R^n)} \norm{D^{\alpha} f}_{L^1(\R^n)}
\leq  \norm{{D^{\alpha}} f}_{L^1(\R^n)},
\end{equation*}
which implies that ${\rm Var}_{\alpha} (\chi_{\Omega} f;\R^n) <\infty$.
\end{proof}
\section{Fractional BV in the Gagliardo sense}\label{s:bvalphagag}
The notion of fractional BV from \Cref{s:fracbvriesz} (as in \cite{CS19}) 
is very similar to the usual $BV$, since it is essentially  a lifting by the Riesz potential.
In this section, we introduce another natural notion, which is denoted by $bv^\alpha$. 
This notion  recovers the fractional perimeter as defined by Caffarelli-Roquejoffre-Savin 
in \cite{MCRS10}. We begin by introducing a different type of fractional divergence as defined in \cite{MS18}. We stress that {related notions} were known before \cite{DGLZ13} and are classically used in the theory of Dirichlet forms, cf. \cite{H15}.

A (nonlocal) vector-field $F$ on $\R^n$ is defined as an $\Lb{n} \times \Lb{n}$-measurable map  $F: \R^n \times \R^n \to \R$, which is additionally antisymmetric, i.e. $F(x,y) = -F(y,x)$. 
As in \cite{MS18} the set of such vector-fields is denoted by $\mathcal{M}(\Ep_{od} \R^n)$, where od stands for off-diagonal and (as in the theory of Dirichlet forms)  $\Ep_{od}$ stands for a sort of one-form (we will not really use this aspect, we recommend the reader to take it as a purely notational choice).

We say that $F \in L^p(\Ep_{od} \R^n)$ if 
 $F \in \mathcal{M}(\Ep_{od} \R^n)$ and 
\[
 \|F\|_{L^p(\Ep_{od} \R^n)} := \brac{\int_{\R^n} \int_{\R^n} \frac{|F(x,y)|^p}{|x-y|^n} dx\, dy}^{\frac{1}{p}} < \infty
\]
for $p \in [1,\infty)$, and 
\[
 \|F\|_{L^\infty(\Ep_{od} \R^n)} := \esssup_{x,y \in \R^n} |F(x,y)|<\infty
\]
for $p=\infty$.
For $\Omega \subset \R^n$, we say $F \in L^p_{00}(\Ep_{od}\Omega)$ 
if $F \in L^p(\Ep_{od} \R^n)$ and $F(x,y) = 0$ for $\Lb{n}$-a.e. $x \in \R^n \sm \Omega$ (and thus for a.e. $y \in \R^n \sm \Omega$).

{The (Gagliardo sense) fractional derivative $\dG$\ ,}
which has similar properties as the gradient of a function, 
takes an $\Lb{n}$-measurable function $f: \R^n \to \R$ into a vector-field
\[
(\dG f)(x,y) :=\frac{f(x)-f(y)}{|x-y|^{\alpha}}.
 \]
Let us remark that {if one was to consider stability as $\alpha \to 1$, then it would make more sense to set 
\[
(\dG f)(x,y) :=(1-\alpha)\frac{f(x)-f(y)}{|x-y|^{\alpha}}.
\]
However, we will not use this definition in the paper, for the simplicity of presentation.
}

The scalar product of two vectorfields $F$ and $G$ is given by 
\begin{equation}\label{product.vectorfield}
( F\cdot G)(x) := \int_{\R^n} \frac{F(x,y) G(x,y)}{|x-y|^n}\, dy.
\end{equation}
The fractional divergence $\divG$ is then the formal adjoint to $-\dG$  
with respect to the $L^2(\R^n)$ scalar product, i.e.
for all  $\varphi \in \Cc^\infty(\R^n)$, we have
\begin{equation} \label{Riesz,Div,int,by,parts}
 \int_{\R^n} \div_\alpha F\, \varphi dx
 := - \int_{\R^n} F \cdot \dG \varphi dx
 =- \int_{\R^n}\int_{\R^n} \frac{F(x,y) (\varphi(x)-\varphi(y))}{\abs{x-y}^{n+\alpha}}dydx.
\end{equation}
The multiplication of a scalar function $f(x)$ and a vector field $F(x,y)$ is defined as:
\begin{equation} \label{product.scalar.vectorfield.eq}
(f F)(x,y) :=\frac{f(x)+f(y)}{2}F(x,y).
\end{equation}
Using \eqref{Riesz,Div,int,by,parts}, we can obtain
the integral formula of $\divG$.
By antisymmetry  $F(x,y) = -F(y,x)$ and the Fubini's theorem, we have
\[
\begin{split}
 \int_{\R^n}\int_{\R^n} F(x,y) \frac{\varphi(x)-\varphi(y)}{\abs{x-y}^{n+\alpha}}dydx = 2\int_{\R^n}\int_{\R^n} \frac{F(x,y)}{\abs{x-y}^{n+\alpha}}dy\, \varphi(x)\, dx, 
\end{split}
\]
which enables us to give the integral definition of $\divG F$ by
\[
(\divG F)(x) := - 2\int_{\R^n}  \frac{F(x,y)}{\abs{x-y}^{n+\alpha}} dy=- \int_{\R^n}  \frac{F(x,y)-F(y,x)}{\abs{x-y}^{n+\alpha}} dy.
\]
In what follows, by yet another slight abuse of notation we are going to use this formulation even when $F(x,y) \neq -F(y,x)$:
\begin{equation}\label{eq:divalpha}
(\divG F)(x) := - \int_{\R^n}  \frac{F(x,y)-F(y,x)}{\abs{x-y}^{n+\alpha}} dy.
\end{equation}
It was shown in \cite{MS18} how this fractional divergence naturally appears and leads to conservation laws and div-curl type results in the theory of fractional harmonic maps. 

With the Fourier transform, one can check that
\begin{equation}
(-\Delta)^{\alpha} f
=
 - c\, {\rm div}_{\alpha} (\dG f )
\end{equation}
for some constant $c=c(n,\alpha)$.

Armed with the fractional divergence $\divG$, we can define the fractional {bounded} variation in the Gagliardo sense. 
\begin{definition}[Gagliardo-type fractional BV]  \label{def:BV}
Let $f \in L^1_{loc}(\R^n)$.
For an open set $\Omega \subset \R^n$, we define
\[
 {\rm var}_\alpha(f;\Omega) := \sup \left \{\int_{\R^n} f \div_\alpha \Phi\, dx: \quad \Phi \in \Cc^\infty(\Omega\times \Omega), \quad \|\Phi\|_{{L^\infty(\R^n\times \R^n)}} \leq 1 \right \} \, .
\]
Observe that this is equivalent to 
\[
 {\rm var}_\alpha(f;\Omega) = \sup \left \{\int_{\Omega} f \div_\alpha \Phi\, dx: \quad \Phi \in \Cc^\infty(\Omega\times \Omega), \quad \|\Phi\|_{{L^\infty(\Omega\times \Omega)}} \leq 1 \right \} \, .
\]
We say that $f \in bv^\alpha(\Omega)$ if 
\[
 \|f\|_{bv^\alpha(\Omega)} := \|f\|_{L^1(\Omega)} + \var_\alpha(f;\Omega) < \infty.
\]
\end{definition}

The notion $\var_\alpha(f;\Omega)$ is well-defined by the following observations.
First, to have consistency, we observe that
\begin{lemma} \label{Gagliardo.div.Lp}
Let $\Phi \in \Cc^1(\R^n\times \R^n)$, 
then for all $\alpha \in (0,1) $ and all  $p \in [1,\infty]$,
we have 
\[
 \div_\alpha \Phi \in L^p(\R^n) .
\]
\end{lemma}
Observe that we exclude the case $\alpha = 1$ since  $\div_\alpha \Phi$ is not well defined for $\alpha =1$.
A multiplication with $(1-\alpha)$ would lead to a stable theory as $\alpha \to 1$.  
\begin{proof}
Observe that by differentiability of $\Phi$,
\[
 |\Phi(x,y) -\Phi(y,x)| \leq |\Phi(x,y) -\Phi(x,x)| + |\Phi(x,x)-\Phi(y,x)| \leq 2\|D\Phi\|_{L^\infty(\R^n \times \R^n)}\, |x-y|.
\]
Then, using a similar argument as in Lemma \ref{la:divalphaphi}, we have
\begin{equation}
\begin{aligned}
\abs{\left(\div_\alpha \Phi \right)(x)} 
&\leq 2 
\brac{\| \Phi  \|_{L^\infty(\R^n\times \R^n)}  +  \|D \Phi \|_{L^\infty(\R^n \times \R^n)} }
\int_{\R^n} \frac{\min\{1,|x-y|\} }{|x-y|^{n+\alpha}} dy
\\
&
\aleq_{\alpha} 
 \brac{\| \Phi  \|_{L^\infty(\R^n \times \R^n)}  +  \|D \Phi \|_{L^\infty(\R^n \times \R^n)} },
\end{aligned}
\end{equation}
which implies that
$
 \|\div_\alpha \Phi\|_{L^\infty(\R^n)} < \infty. 
$
It remains to prove that $ \|\div_\alpha \Phi\|_{L^1(\R^n)}<\infty$.
Then the required result can be obtained using interpolation.
Since $\Phi$ is compactly supported, we may suppose 
${\rm supp } \, {\Phi} \subseteq {B(0,M)\times B(0,M)}$
for some $M > 0$. 
Thus, we obtain
\begin{equation} \label{dsf.is.L1}
\begin{aligned}
\norm{\div_\alpha \Phi}_{L^1(\R^n)}&=
\int_{{B(0,M)}} \abs{ \int_{{B(0,M)}} \frac{\Phi(x,y) - \Phi(y,x)}{\abs{x-y}^{n+\alpha} } dy } dx
\\
&
\aleq  \| D \Phi  \|_{L^\infty(\R^n \times \R^n)} \int_{{B(0,M)}} \int_{{B(0,M)}} \frac{1}{\abs{x-y}^{n+\alpha -1} } dy  dx < \infty,
\end{aligned}
\end{equation}
which finishes the proof.
\end{proof}
We introduce the definition of space $W^{\alpha,1}(\Omega)$, see \cite{MS18} for details.
\begin{definition}
Let $\Omega\subseteq\mathbb{R}^n$ be an open set.
A function $f$ is in $W^{\alpha,1}(\Omega)$ when $f\in L^1(\Omega)$ and
\begin{equation}
[f]_{W^{\alpha,1}(\Omega)}:= \int_{\Omega}\int_{\Omega} \frac{ \abs{f(x)-f(y) } }{ \abs{x-y}^{n+\alpha} } dydx <\infty. 
\end{equation} 
The norm  of $W^{\alpha,1}(\Omega)$  
is defined as
\begin{equation}
\norm{f}_{W^{\alpha,1}(\Omega)}:= \norm{f}_{L^1(\Omega)} + [f]_{W^{\alpha,1}(\Omega)}.
\end{equation}

\end{definition}
We now state our main theorem of this section, which is {in strong contrast to} the Riesz-type fractional $BV$ functions, cf. \Cref{la:Halpha1}. The fractional $BV$ space $bv^\alpha$ is actually equivalent to $W^{\alpha,1}$,
which makes this space  more tractable and probably more attainable for numerical purposes.
\begin{theorem}\label{th:bvsws1}
Let $\alpha \in (0,1)$. 
Let $\Omega \subseteq \R^n$ be any open set. Then $bv^\alpha(\Omega) = W^{\alpha,1}(\Omega)$. More precisely, for any $f \in L^1(\Omega)$ we have
\[
 \var_\alpha(f;\Omega) = [f]_{W^{\alpha,1}(\Omega)},
\]
whenever one of the two sides {are finite}.
\end{theorem}
\begin{remark}It is well known that \Cref{th:bvsws1} is false for $\alpha = 1$: e.g. take any nonempty open and bounded set $\Omega$ with finite perimeter. Then $\chi_{\Omega} \not \in W^{1,1}(\R^n)$ (e.g. because it is not continuous on almost all lines). 
However, we have $\chi_{\Omega} \in BV(\R^n)$. 
In that sense, \Cref{th:bvsws1} may be surprising at first. Let us mention that, although we are not aware of \Cref{th:bvsws1} in the literature, intuitively related observations have been made by people working with fractional perimeters.
\end{remark}
\begin{remark}
An immediate corollary is that the fractional perimeter as defined by Caffarelli-Roquejoffre-Savin, \cite{MCRS10}, $\Per_\alpha(\Omega;\R^n) = \var_\alpha(\chi_{\Omega};\R^n)$. 
Thus, the space
$bv^\alpha$ is the naturally associated notion for a fractional $BV$ space when working with that perimeter.
\end{remark}

We prove several lemmas before proving \Cref{th:bvsws1}.

\begin{lemma} \label{f.in.Wa1.equal}
Suppose $f \in W^{\alpha,1}(\Omega)$, then we have $f\in bv^{\alpha}(\Omega)$ and
$ \var_\alpha(f;\Omega) = [f]_{W^{\alpha,1}(\Omega)} $.
\end{lemma}
\begin{proof}
Given any $\Phi\in \Cc^{1} (\Omega\times \Omega,\R)$. 
Without loss of generality, we may suppose ${\rm supp} \, \Phi \subseteq K\times K$, while $K$ is a compact subset of $\Omega$.
Then we obtain that also ${\rm div}_{\alpha} \Phi= 0$ outside $K$.

We have from Fubini's theorem (since $f \in W^{\alpha,1}(\Omega)$, both sides converge absolutely)
\begin{equation}\label{eq:sldkgjfsfksdj}
\int_{\mathbb{R}^n}f  {\rm div}_{\alpha} \Phi dx 
=
 -  \int_{\Omega}\int_{\Omega} 
\frac{f(x)-f(y)}{\abs{x-y}^\alpha}  \Phi(x,y) \frac{ dy dx}{\abs{x-y}^n}.
\end{equation}
Since 
$L^\infty(\Omega \times \Omega)$ is the dual of $L^1(\Omega \times \Omega)$,
 from \eqref{eq:sldkgjfsfksdj} we obtain 
\[
{\rm var}_{\alpha}( f  ;\Omega) 
=[f]_{W^{\alpha,1}({\Omega})},
\]
{which completes the proof.}
\end{proof}

The lemma above has not yet proven \Cref{th:bvsws1}: if we only know $f \in bv^\alpha(\Omega)$ we cannot yet apply \Cref{f.in.Wa1.equal}.
However, \Cref{f.in.Wa1.equal} does give us the direction $\var_\alpha(f;\Omega) \leq [f]_{W^{\alpha,1}(\Omega)}$ 
whenever the right-hand side is finite (because in that case we can indeed apply \Cref{f.in.Wa1.equal}). 

Next, we observe the following lower semi-continuity result.

\begin{lemma}\label{la:bvlowersem}
Suppose $f_k \in bv^{\alpha}(\Omega)$ {for all $k \in \mathbb{N}$} 
and $\norm{f_k - f}_{L^1{(\Omega)}}\to 0$ as $k\to \infty$.
Then we have
\[
 \var_\alpha(f;\Omega) \leq  \liminf_{k \to \infty} \var_\alpha(f_k;\Omega),
\]
and in particular $f \in bv^\alpha(\Omega)$.
\end{lemma}
\begin{proof}
Consider any  $\Phi\in \Cc^1(\Omega\times\Omega)$ with $\norm{\Phi}_{{L^\infty(\Omega\times \Omega)}}\leq 1$.
Since $f_k\to f$ in $L^1(\Omega)$, 
and ${\rm div}_{\alpha}\Phi$ is bounded by Lemma \ref{Gagliardo.div.Lp},
we have $\norm{f_k{\rm div}_{\alpha}\Phi - f {\rm div}_{\alpha}\Phi}_{L^1{(\Omega)}}\to 0$.
Thus, we have
\begin{equation}
\int_{\Omega} f {\rm div}_{\alpha}\Phi dx
= \lim_{k\to\infty}  \int_{\Omega} f_k {\rm div}_{\alpha}\Phi dx 
\leq \liminf_{k\to\infty} {\rm var}_{\alpha}(f_k;\Omega).
\end{equation}
Taking the supremum over all admissible $\Phi$, we have
\begin{equation}
{\rm var}_{\alpha}(f;\Omega)
\leq \liminf_{k\to\infty} {\rm var}_{\alpha}(f_k;\Omega),
\end{equation}
{which completes the proof.}
\end{proof}

To prove $\var_\alpha(f;\Omega) \geq [f]_{W^{\alpha,1}(\Omega)}$ 
(whenever the left-hand side is finite),
 the last missing ingredient is the following recovery sequence result. 
In the following we say  that a set $G$ is compactly contained in a set $\Omega$, in symbols $G \subset \subset \Omega$, if $G$ is bounded and $\overline{G} \subset \Omega$.
 \begin{lemma}  \label{var.a.approximate}
    Let $\Omega \subseteq \R^n$ {be any} open set. Assume $f \in L^1(\Omega)$ with $\var_\alpha(f;\Omega) < \infty$. Then for any open $G \subset \subset \Omega$ there exists $f_k \in \Cc^\infty(\Omega)${, for all $k \in \mathbb{N}$,} such that 
    \[
    f_k \to f \quad \text{in $L^1(G)$}
    \]
    and
    \[
    \limsup_{k \to \infty}\var_\alpha(f_k;G) \leq \var_\alpha(f;\Omega). 
    \]
\end{lemma}
\begin{proof}
Since $G\subset\subset \Omega$, there exist open sets $U$ and $V$ 
such that $G\subset\subset U \subset\subset V \subset\subset \Omega$.
Pick $\zeta\in \Cc^{\infty}(V)$ 
such that $\zeta = 1$ on $U$ and $\zeta \leq 1$ in all of $\R^n$. 
Take $\eps_0 > 0$ such that $B_{\varepsilon}(G) := \{z \in \R^n: \dist(z,G) < \eps\} \subset\subset U$ 
and  $B_{\varepsilon}(V)\subset\subset \Omega$ for any $\eps \in (0,\eps_0)$.
Let $\eta \in \Cc^\infty(B(0,1))$, $\int \eta = 1$, be the usual mollifier kernel and set $\eta_\eps := \eps^{-n} \eta(\cdot/\eps)$. 
For $\eps \in (0,\eps_0)$, we define 
$f_{\varepsilon}
:=
\eta_{\varepsilon} * (f \zeta)
$,
then ${\rm supp} f_{\varepsilon}  \subseteq \Omega$. 
Given any $\Phi \in C_c^1(G\times G)$ with $\|\Phi\|_{L^\infty(\R^n \times \R^n)} \leq 1$. 
Using \eqref{eq:divalpha}, the Fubini's theorem and the substitution $x' = x-z$ and $y' = y-z$, we obtain
\begin{equation} \label{var,estimate,f,epsilon}
\begin{aligned}
& 
\int_{\mathbb{R}^n} f_{\varepsilon}\,  {\rm div}_{\alpha} \Phi  dx\\
&
= 
\int_{G} \eta_{\varepsilon} * (f \zeta) {\rm div}_{\alpha}\, \Phi  dx
\\
&
=
\int_G
\brac{\int_{{B(0,\varepsilon)}} \eta_{\varepsilon}(z) \, f(x-z)\, \zeta(x-z) dz }
\brac{-\int_G \frac{ \Phi(x,y)-\Phi(y,x)  }{ \abs{x-y}^{n+\alpha}  } dy } dx
\\
&
=
- \int_{ G } \int_{G} \int_{{B(0,\varepsilon)}} f(x-z)\, \zeta(x-z)\, \eta_{\varepsilon}(z) 
\, \frac{ \Phi(x,y)-\Phi(y,x)  }{ \abs{x-y}^{n+\alpha}  } dy\, dx\, dz
\\
&
=
- \int_{B_{\varepsilon}(G) } \int_{B_{\varepsilon}(G) } 
\int_{ {B(0,\varepsilon)} }
 f(x')\, \zeta(x')\,  \eta_{\varepsilon}(z) 
\,  \frac{ \Phi(x'+z,y'+z)-\Phi(y'+z,x'+z)  }{ \abs{x'-y'}^{n+\alpha}  } dz\, dy'\, dx'.
\\
\end{aligned}
\end{equation}
Notice that since $\eta_{\varepsilon}(-z)= \eta_{\varepsilon}(z)$, we have
\begin{equation}\label{eq:Gagmollifier}
 \left(\eta_{\varepsilon}*  \Phi\right)(x',y') := \int_{ {B(0,\varepsilon)} } \eta_{\varepsilon}(z) 
\, \brac{ \Phi(x'+z,y'+z)-\Phi(y'+z,x'+z)  } dz.
\end{equation}
Thus, by \eqref{var,estimate,f,epsilon} we have
\[
\begin{split}
\int_{\mathbb{R}^n} f_{\varepsilon}  {\rm div}_{\alpha} \Phi  dx
=&
- \int_{B_{\varepsilon}(G) } \int_{B_{\varepsilon}(G) } 
 f(x')\, \zeta(x')\,  
\,  \frac{ \left(\eta_{\varepsilon}*  \Phi\right)(x',y')  }{ \abs{x'-y'}^{n+\alpha}  } dy'dx'
\\
=&- \int_{B_{\varepsilon}(G) } \int_{B_{\varepsilon}(G) }  f(x')\, \frac{1}{2} (\zeta(x')+\zeta(y'))\,  \frac{ \left(\eta_{\varepsilon}*  \Phi\right)(x',y')  }{ \abs{x'-y'}^{n+\alpha}  } dy'dx'\\
&- \int_{B_{\varepsilon}(G) } \int_{B_{\varepsilon}(G) }  f(x')\, \frac{1}{2} (\zeta(x')-\zeta(y'))\,  \frac{ \left(\eta_{\varepsilon}*  \Phi\right)(x',y')  }{ \abs{x'-y'}^{n+\alpha}  } dy'dx'.\end{split}
\]
Since $B_\eps(G) \subset U$ and $\zeta \equiv 1$ in $U$, the second term vanishes.
Setting 
\[
 \Psi_\eps (x',y') := \frac{1}{2} (\zeta(x')+\zeta(y'))\,  \left(\eta_{\varepsilon}*  \Phi\right)(x',y'), 
\]
we see that $ \Psi_\eps \in \Cc^\infty(\Omega \times \Omega)$,
$\Psi_\eps(x',y') = - \Psi_\eps(y',x')$,
 and  
\[
\int_{\mathbb{R}^n} f_{\varepsilon}\,  \div_{\alpha} \Phi  dx
= \int_{\Omega }  f\, \div_{\alpha} \Psi_{\eps}\, dx'.
\]
It is easy to check that 
\[
 \abs{\frac{1}{2} (\zeta(x')+\zeta(y'))\,  \left(\eta_{\varepsilon}*  \Phi\right)(x',y') } 
  \leq \abs{\left(\eta_{\varepsilon}*  \Phi\right)(x',y')} \leq \|\Phi \|_{L^\infty(\R^n \times \R^n)} \leq 1.
\]
Thus, we have  shown that for any $\Phi \in \Cc^\infty(G \times G)$ with $\|\Phi \|_{L^\infty(\R^n \times \R^n)} \leq 1$, and any $\eps < \eps_0$, there is
\[
\int_{\mathbb{R}^n} f_{\varepsilon}  {\rm div}_{\alpha} \Phi  dx \leq  \var_\alpha(f;\Omega).
\]
Taking the supremum over such test-functions $\Phi$ we obtain 
\[
 \sup_{\eps \in (0,\eps_0)} \var_\eps(f_\eps;G) \leq  \var_\alpha(f;\Omega).
\]
In particular, 
\[
 \limsup_{\eps \to 0} \var_\eps(f_\eps;G) \leq  \var_\alpha(f;\Omega).
\]
By usual mollifier arguments we have $f_{\varepsilon} \to \zeta f$ in $L^1(\R^n)$ as $ \varepsilon\to 0$.
Since $\zeta \equiv 1$ in $G$,
we have $f_{\varepsilon} \to f$ in $L^1(G)$ as $ \varepsilon\to 0$.
\end{proof}

We now finish the proof of the main theorem.

\begin{proof}[Proof of \Cref{th:bvsws1}]
Let $f \in L^1(\Omega)$. 
In \Cref{f.in.Wa1.equal}, {we have proved} 
$
 [f]_{W^{\alpha,1}(\Omega)} \geq \var_{\alpha}(f;\Omega),
$
whenever the left-hand side is finite.
So we only need to establish
$
 [f]_{W^{\alpha,1}(\Omega)} \leq \var_{\alpha}(f;\Omega),
$
whenever the right-hand side is finite. 

Given any $G \subset \subset \Omega$, 
we can take a sequence $\{f_k\}_{k=1}^{\infty}$ as stated in Lemma \ref{var.a.approximate}. 
Since $f_k\in \Cc^{\infty}(\Omega)$,
we have
$f_k \in W^{\alpha,1}(G)$, so Lemma~\ref{f.in.Wa1.equal} is applicable.
Combining Lemma \ref{f.in.Wa1.equal} and \Cref{var.a.approximate}, we find
\[
 \limsup_{k \to \infty} [f_k]_{W^{\alpha,1}(G)} \leq \limsup_{k \to \infty}\var_{\alpha}(f_k;G) \leq \var_{\alpha}(f;\Omega).
\]
Since $f_k\to f$ in $L^1(G)$, up to passing to a subsequence, we may assume that $f_k(x)$ converges to $f(x)$ a.e. in $G$.
Using Fatou's lemma, we obtain
\begin{equation}
\begin{aligned}
\int_G \int_G \frac{ \abs{f(x) - f(y)} }{\abs{x-y}^{n+\alpha}  }dydx
\leq 
\liminf_{k\to\infty}
\int_G \int_G \frac{ \abs{f_k (x) - f_k (y)} }{\abs{x-y}^{n+\alpha}} dydx.
\end{aligned}
\end{equation}
Thus, we obtain 
\begin{equation} \label{Wa1.leq.var.compact}
 [f]_{W^{\alpha,1}(G)} \leq \liminf_{k \to \infty} [f_k]_{W^{\alpha,1}(G)} 
 \leq {\rm var}_{\alpha}(f;\Omega). 
\end{equation}

Picking an increasing sequence of open sets $\{ G_m \}$,
 such that $G_m\subset\subset \Omega$ and
\begin{equation}
\bigcup_{m=1}^{\infty} G_m = \Omega.
\end{equation}
Applying the above argument to $G = G_m$, we have $[f]_{W^{\alpha,1}(G_m) } \leq {\rm var}_{\alpha}(f;\Omega) $ for any $m\in\mathbb{N}$.
{Using Fatou's lemma again}, we have 
\begin{equation}
\begin{split}
[f]_{W^{\alpha,1}(\Omega) }
&
\leq
\liminf_{m\to\infty}
\int_{G_m} \int_{G_m} \frac{ \abs{f(x) - f(y)} }{\abs{x-y}^{n+\alpha}  }dydx
\\
&\leq \liminf_{m\to\infty} [f]_{W^{\alpha,1}(G_m) }
\leq {\rm var}_{\alpha}(f;\Omega),
\end{split}
\end{equation}
which concludes the proof.
\end{proof}

Using \Cref{th:bvsws1}, we can easily obtain the following result.
\begin{proposition}[Weak compactness] \label{weak_compGagl}
Let $\Omega \subset \R^n$ be an open and bounded set with Lipschitz boundary. 
Assume that $\{f_k\}_{k=1}^{\infty} \subset bv^\alpha(\Omega)$ such that
\[
\sup_{k \in \N} \|f_k\|_{bv^\alpha(\Omega)} < \infty.
\]
Then there exists $f \in bv^\alpha(\Omega)$ such that 
\[
 \var_\alpha(f;\Omega) \leq \liminf_{k \to \infty} \var_\alpha(f_k;\Omega),
\]
and there is a subsequence $\{f_{k_i}\}_{i=1}^{\infty}$,
 such that for all $ p \in \left[1,\frac{n}{n-\alpha}\right)$ 
\[
 \|f_{k_i} - f\|_{L^p(\Omega)} \xrightarrow{i \to \infty} 0 .
\]
\end{proposition}
\begin{proof}
By \Cref{th:bvsws1}, we have
\[
 \var_\alpha(f_k,\Omega) = [f_k]_{W^{\alpha,1}(\Omega)}.
\]
Since $\Omega$ is a Lipschitz domain, it is regular in the sense of \cite{Z15}.
Thus, by the main result of \cite{Z15}, we can find an extension $\tilde{f}_k \in W^{\alpha,1}(\R^n)$ 
with compact support, $\tilde{f}_k = f_k$ a.e. in $\Omega$, such that 
\[
 [\tilde{f}_k]_{W^{\alpha,1}(\R^n)} \aleq [f_k]_{W^{\alpha,1}(\Omega)}.
\]
From the usual Rellich theorem, we find a subsequence $(f_{k_i})_{i \in \N}$,
such that for all $ p \in \left[1,\frac{n}{n-\alpha}\right)$ 
\[
 \|f_{k_i} - f\|_{L^p(\Omega)} \xrightarrow{i \to \infty} 0 .
\]
see \cite[Corollary 7.2]{Hitchhiker}. In particular, in view of \Cref{la:bvlowersem}, 
\[
 \var_\alpha(f;\Omega) \leq \liminf_{k \to \infty} \var_\alpha(f_k;\Omega).
\]
\end{proof}
Using \Cref{th:bvsws1}, we also readily obtain the Sobolev embedding theorem,
which can be proved 
using the extension theorem as in \Cref{weak_compGagl} above and then \cite[Theorem 9]{L14}.
\begin{proposition}
 Let $\Omega \subset \R^n$ be an open and bounded set with Lipschitz boundary. Then there exists a constant $C = C(n,\alpha) > 0$ such that for any $f \in bv^\alpha(\Omega)$,
 \[
  \|f\|_{L^{\frac{n}{n-\alpha}}(\Omega)} \leq C\, \var_\alpha(f;\Omega).
 \]
\end{proposition}

We also  obtain the following density result, which might be known to experts (observe that this density is not true for $\alpha =1$, cf. \cite[Theorem 5.3 and remark after]{EG15}).
Using the identification in \Cref{th:bvsws1}, the extension property in \cite{Z15} , and the usual mollifcation in \cite[Theorem 2.4.]{Hitchhiker} or \cite[Lemma 26]{Mironescu05}, 
we have the following result.
\begin{corollary} Let $\alpha \in (0,1)$.
Let $\Omega \subset \R^n$ be any open and bounded set with Lipschitz boundary, then $\Cinf(\overline{\Omega})$ is dense in $bv^\alpha(\Omega)$.
\end{corollary}

Let us make a last remark about traces. For a classical $BV$ function there is a trace, \cite[Theorem 10.2.1]{Attouch2014}. 
However, this will not be true for $bv^\alpha(\Omega)$, 
since $W^{\alpha,1}(\Omega)$ does not have a reasonably defined trace. The typical approach is then the notion of a fat boundary trace, which we do not pursue in this paper.


\section{Image Denoising and Predual Problem}\label{s:imagedenoise}
Let $\Omega \subset \R^n$ be a open and   bounded set with a Lipschitz continuous boundary, $\alpha \in (0,1)$, $p \in (1,\infty)$, $p^\infty:= \frac{n}{n-\alpha}$, and $u_N\in L^p(\Omega)$. Based on the  two fractional variations considered in this work we consider the (primal) problems for some fixed positive parameters $\beta$ and $\gamma$
\begin{align}
 \qquad &\inf_{u\in L^p(\Omega)} \left\{ \frac{\gamma}{p}\|u-{u}_N\|_{L^p(\Omega)}^p + \beta \VR\left(\chi_\Omega u;\R^n \right) \right\},
 \tag{$\mathscr{P}_R$}\label{def:primalProblem1}\\
 \qquad &\inf_{u\in L^p(\Omega)} \left\{ \frac{\gamma}{p}\|u-{u}_N\|_{L^p(\Omega)}^p + \beta  \VG (u;\Omega) \right\}.
 \tag{$\mathscr{P}_G$}\label{def:primalProblem2}
\end{align}
Note that the condition of $u$ having bounded fractional variation is imposed  implicitly, and
it is also clear that both problems are strictly convex for $p>1$. Therefore, we use
 well-known results from {\it convex analysis}, cf. \cite{Ekeland99}, to study the  minimizers  of Problems (\ref{def:primalProblem1}) and (\ref{def:primalProblem2}).
The
regularity theory to a related problem to $\mathscr{P}_G$ was recently studied in \cite{NovagaOnue21}

\subsection*{Convex Analysis and Optimization}\label{subsec:convex}
As  usual in convex optimization,  we consider the so-called dual problem, which usually gives new insights about the structure of the primal  problem. In this work, we consider a different but related approach coined as predual method. Here  we mainly follow the approach given in \cite{Burger,Chavent1997,Hintermuller}. 
In order to introduce this method, we need some definitions, 
 cf.  \cite[Ch. I]{Ekeland99}. Consider a Banach space $V$  and its topological dual $V^*$, with duality paring denoted by $\displaystyle \langle\cdot,\cdot\rangle_{V^*,V}$. Given  $\mathscr{F}:V \rightarrow  \overline{\R}$, its {\it Fenchel conjugate} is given  by  $\mathscr{F}^*: V^* \rightarrow \overline{\R}$,
 \begin{align}
 u^*\mapsto \mathscr{F}^*(u^*):= \sup_{u\in V} \left\{ \langle u^*,u\rangle_{V^*,V}-\mathscr F(u) \right\}. \label{def:FenchelC}
 \end{align}
 We denote by {$\partial \mathscr{F}(u)$}  the subdifferential map of {$\mathscr{F}$} at the point $u\in V,$ see  \cite[Definition I.5.1]{Ekeland99}. The following characterization holds,
 %
 {
 \begin{align}
 \begin{aligned}
    u^*\in \partial \mathscr{F}(u)  \mbox{ if and only if } \mathscr{F}(u)  \mbox{ is finite and }\\
    \left\langle u^*, v-u\right\rangle_{V^*,V} + \mathscr{F}(u) \leq \mathscr{F}(v),\qquad \forall v\in V.
 \end{aligned}\label{def:subdiff}
 \end{align}}
 %
 We now introduce  a  process known as {\it dualization} \cite[Chs. III-IV]{Ekeland99}, here we 
 will focus on problems of the form: 
 \begin{align} \tag{$\mathscr{Q}$}\label{def:Q}
 \qquad  \inf_{u\in V} \{  F(u)+G(\Lambda u)\},
 \end{align}
 where 
 $Y$ is a Hausdorff topological space with dual $Y^*$, $\Lambda \in \mathscr{L}(V,Y)$, with transpose 
  $\Lambda^* \in \mathscr{L}(Y^*,V^*)$, and $F : V \rightarrow \overline{\R}$, $G : V \rightarrow \overline{\R}$. 
  We define  the dual problem of (\ref{def:Q}) as
 \begin{align} \tag{$\mathscr{Q}^*$}\label{def:Q*}
  \qquad \sup_{v\in Y^*} -\Phi^*(0,v),
 \end{align}
 where $\Phi^*:V^*\times Y^* \rightarrow \overline{\R}$ is the Fenchel conjugate (dual) of $\Phi:V\times Y \rightarrow\overline{\R}, (u,p)\mapsto {\Phi(u,p)} := F(u) + G(p+ \Lambda u)$, see \eqref{def:FenchelC}.
 The next theorem gives conditions for the so-called {\it Fenchel's duality}, cf. \cite[Theorem III.4.1]{Ekeland99} and \cite[Pg.~130]{IEkeland_TTurnbull_1983a}. 
 %
 \begin{theorem} \label{theorem:FenchelTheorem} Assume $V$ and $Y$ are  Banach spaces, $F$ and $G$ are convex and lower semicontinuous (l.s.c.), and there exists  $v_0\in V$ such that $F(v_0)<\infty,$ $G(\Lambda v_0)<\infty,$ and $G$ is continuous at $\Lambda v_0$. 
 Then, the problems  (\ref{def:Q}) and   (\ref{def:Q*}) are related by:
 \begin{align*}
 \inf_{u\in V} \{  F(u)+G(\Lambda u)\}  =& \sup_{v \in Y^*} -\Phi^*(0, v ) \\
                  =& \sup_{v \in Y^*} \left\{ - F^*(\Lambda^* v)- G^*(-v) \right\} ,
 \end{align*}
and there exists at least one solution to \eqref{def:Q*}.
 Moreover, if $\overline{u}$ and $\overline{v}$ are solutions for (\ref{def:Q}) and (\ref{def:Q*}), respectively, then
 \begin{align}
    \begin{aligned}
        \Lambda^*\overline{v} \in \partial F(\overline{u}),\\
        -\overline{v} \in \partial G(\Lambda \overline{u}).
    \end{aligned}\label{def:OptCond}
 \end{align}
 \end{theorem}

 In general, there are different choices for $F,G$ and $\Lambda$ in order to write a given  problem  as 
 in (\ref{def:Q}). Here  we consider one that satisfy the hypothesis of Theorem \ref{theorem:FenchelTheorem} 
 in a straightforward manner.
We now show existence and characterization for minimizers of problems (\ref{def:primalProblem1}) and 
(\ref{def:primalProblem2}).
 \subsection*{Riesz-type}
By Proposition \ref{Prop:embeddingLp}, for $\displaystyle p\in \left[1,p^\infty \right]$, with $p^\infty:= \frac{n}{n-\alpha}$, 
we can consider the problem (\ref{def:primalProblem1}) def{}ined on  $L^p(\Omega)$ or $\BVRz$, cf. (\ref{def:BV00}), 
interchangeably. 
The next lemma shows that the problem  (\ref{def:primalProblem1}), related to the Riesz-type of fractional 
bounded variation, has a solution and for $p>1$ it is unique. 
%
\begin{lemma}\label{lemma:exisUniq-I}
    For $\displaystyle p \in \left(1,p^\infty\right)$, the problem   (\ref{def:primalProblem1})  has a unique  solution $\overline{u}\in \BVRz$ 
\end{lemma} 
%
\begin{proof} Let $p\in [1,\infty)$,  def{}ine $\mathscr{J}_R: \left( L^p(\Omega), \|\cdot \|_{L^p(\Omega)}\right) \rightarrow \overline{\R},$ given by 
\begin{align}
   \mathscr{J}_R(u):=\frac{\gamma}{p}\| u-u_N\|^p_{L^p(\Omega)} +\beta \VR(\chi_\Omega u;\R^n). \label{def:F}
\end{align} 
It is clear that 
$$
0\leq \inf_{u\in L^p(\Omega)}  \mathscr{J}_R(u) \leq  \frac{\gamma}{p}\| u_N\|^p_{L^p(\Omega)}. 
$$
Now, let $(u_k)_{k\in \N}\subseteq L^p(\Omega)$ be a minimizing sequence associated to the problem (\ref{def:primalProblem1}), 
then for each $k\in \mathbb{N}$ 
\begin{align*}
 &\|u_k\|_{L^p(\Omega)} \leq \|u_k-u_N\|_{L^p(\Omega)} +\|u_N\|_{L^p(\Omega)} \leq 2\|u_N\|_{L^p(\Omega)}, \mbox{ and }\\
 &\VR(\chi_\Omega u_k;\R^n)\leq \frac{\gamma}{ p\beta} \|u_N\|_{L^p(\Omega)}.
\end{align*}
Then, for   $ p \in \left[ 1, p^\infty \right)$, Propositions \ref{Prop:embeddingLp}  and \ref{weak_comp} imply  there exist $\overline{u}\in \BVRz \hookrightarrow L^p(\R^n)$  and   a subsequence $\{u_{k_i}\}_{i \in \N}$ such that  
\begin{align*}
 \VR(\overline{u};\R^n) \leq \liminf_{i \to \infty} \VR(u_{k_i};\R^n) \quad \mbox{and} \quad 
  \|u_{k_i} - u_N\|^p_{L^p(\Omega)}  \xrightarrow{i \to \infty} \|\overline{u} - u_N\|^p_{L^p(\Omega)}. 
\end{align*}
Thus, the existence of a solution for (\ref{def:primalProblem1}) follows from the fact that $\overline{u} = \chi_\Omega \overline{u}$, a.e., for the uniqueness it is enough to notice  that $\displaystyle \mathscr{J}_R$, cf. (\ref{def:F}), is a strictly convex functional for $p>1$. In fact, if  $\overline{u}_1$ and $\overline{u}_2$ were  two different solutions  to (\ref{def:primalProblem1}), then for  $\lambda\in (0,1)$, 
\begin{align*}
    \mathscr{J}_R(\lambda \overline{u}_1 +(1-\lambda)\overline{u}_2) =& \frac{\gamma}{p}\|\lambda \overline{u}_1 +(1-\lambda)\overline{u}_2-u_d\|^p_{L^p(\Omega)}
    +\beta \VR(\chi_\Omega (\lambda \overline{u}_1 +(1-\lambda)\overline{u}_2);\R^n)\\
    =& \frac{\gamma}{p}\|\lambda (\overline{u}_1-u_d) +(1-\lambda)(\overline{u}_2-u_d)\|^p_{L^p(\Omega)}  \\ 
    &{}+\beta \VR(\chi_\Omega (\lambda \overline{u}_1 +(1-\lambda)\overline{u}_2);\R^n)\\
    <& \frac{\gamma}{p} \lambda\| \overline{u}_1-u_d\|^p_{L^p(\Omega)} + \frac{\gamma}{p}(1-\lambda)\|\overline{u}_2-u_d\|^p_{L^p(\Omega)}\\
     &{}+\beta \VR(\chi_\Omega (\lambda \overline{u}_1 +(1-\lambda)\overline{u}_2);\R^n)\\
    \leq &  \lambda\mathscr{J}_R( \overline{u}_1 ) + (1-\lambda)\mathscr{J}_R(\overline{u}_2).
\end{align*}
Thus, $\overline{u}_1=\overline{u}_2$ a.e. and by the definition of $\VR$\,, cf. (\ref{Def:Valpha}),  the proof concludes. 
\end{proof}

Next, we will derive an expression of the predual of (\ref{def:primalProblem1}).
In order to do that, we start with the regularity for the ``test functions'' in 
(\ref{Def:Valpha}). It is clear that if $u\in L^1(\R^n)$ and  $\supp  u \subset \Omega$, then   $\int_{{\R^n}} u\, \divR  \phiR\, d x$ does not depend on $\divR \phiR\arrowvert_{\Omega^c}$. This motivates us to define
$\displaystyle  \|\phiR\|_{X_{\text{Riesz}}} := \sqrt[q]{ \|\phiR\|_{L^q(\R^n,\R^n)}^{q} + \|\divR \phiR\|_{L^q(\Omega)}^{q} }$
for $\phiR\in \Cc^1(\R^n;\R^n)$, where  $q:= \frac{p}{p-1}$. 
We consider the space $X_{\text{Riesz}} := X_{\text{Riesz}}(\Omega,q,\alpha),$ given by
  \begin{align*}
     X_{\text{Riesz}} &:= \overline{\Cc^1(\R^n;\R^n)}^{\|\cdot\|_{X_{\text{Riesz}}}} .
 \end{align*}
{We also define an auxiliary problem}   
\begin{align}
\qquad \inf_{\phiR \in  X_{\text{Riesz}}}
\left\{ \frac{1}{q} \|-\divR \phiR\|^{q}_{L^q(\Omega)}- \int_{\Omega} u_N (-\divR \phiR)
 + I_{\beta}(\phiR) \right\}, \tag{$\mathscr Q_R$} \label{def:predual}
\end{align}
where {$I_\beta$ denotes the convex indicator function defined as}
\begin{align*}
    I_{\beta}(\phiR):=
    \begin{cases}
        0 &: \| \phiR\|_{L^\infty(\R^n)} \leq \beta,\\
        +\infty &: \mbox{otherwise.}
    \end{cases}
\end{align*}
{We will establish that \eqref{def:predual} is the pre-dual problem to \eqref{def:primalProblem1},
i.e., dual of \eqref{def:predual} will be \eqref{def:primalProblem1} if $\Omega$ is convex.}

{We begin by noticing that} (\ref{def:predual}) {fits in the abstract framework of} (\ref{def:Q}) 
if we consider the spaces:  $Y:=\left(L^{q}(\Omega), \|\cdot \|_{L^q(\Omega)}\right)$, 
${V:=\left(X_{\text{Riesz}},  \|\cdot \|_{X_{\text{Riesz}}} \right),}$ and the operators:  
\begin{align}
\begin{aligned}
 &G:Y  \rightarrow \overline{\R}, \quad 
 G(v):= \frac{1}{q} \|v\|_{L^q(\Omega)}^{q} - \int_{\Omega} u_N vdx,\\[0.2cm]
 &F:V  \rightarrow \overline{\R},\quad
 F(\phiR):= I_{\beta}(\phiR),\\[0.2cm]
 &\Lambda : V \rightarrow  Y,\quad
 \Lambda(\phiR) := \left(-\divR \phiR\right)\!|_\Omega.
 \end{aligned} \label{def:FandG}
\end{align}
To compute the dual  problem of (\ref{def:predual}), we compute the Fenchel conjugate of $F,G$ and $\Lambda$, given in (\ref{def:FandG}). 
%

\begin{proposition}\label{lemma:F*andG*} 
Let $\Omega\subseteq\mathbb{R}^n$ be open, bounded, convex. Let 
$Y:=\left(L^{q}(\Omega), \|\cdot \|_{q,\Omega}\right),$ 
${V:=\left(X_{\text{Riesz}},  \|\cdot \|_{X_{\text{Riesz}}} \right)}$, and  
operators $F, G$ and $\Lambda$ be defined as in (\ref{def:FandG}), 
then
\begin{align*}
    G^*&: Y^* \rightarrow \overline{\R},\quad  u \mapsto \frac{1}{p} \|u+u_N\|_{L^p(\Omega)}^p,\\
     F^*&: V^*  \rightarrow \overline{\R},\quad \Psi^*\mapsto  \sup_{\substack{\phiR\in V\\ \|\phiR\|_{L^\infty(\R^n)} \leq \beta}} \langle  \Psi^*,\phiR \rangle_{V^*,V},\\
     F^*\!\circ \Lambda^*&: Y^*  \rightarrow \overline{\R},\quad u  \mapsto \beta \VR(\chi_\Omega u;\R^n).
\end{align*}
\end{proposition}

\Cref{lemma:F*andG*}, while in principle looking very similar to the arguments in \cite[Section 2]{Hintermuller}, contains a quite serious subtlety. Observe that \cite{Hintermuller} does not consider test-functions with the natural restriction $\|\Phi\|_{{L^\infty(\R^n)}} \leq \beta$, but resorts to discussing component-wise control $|\Phi^i| \leq \beta$, $i=1,\ldots,n$ (leading to a nonstandard $BV$-space) -- which is crucially needed in their argument to compute the predual. 

Instead we show in our paper that for bounded, open, convex sets $\Omega$ we do not need such unnatural restrictions. The main property we use is the controllable distance of rescaled sets, cf. \Cref{la:starshapeddist} and \Cref{la:starshapeddist2}. The main novelty is contained in the next proposition. Notice that such a result is even critical to prove the result in \cite{Hintermuller,MR3710331}, where $\alpha = 1$, for the natural $BV$-space.

\begin{proposition}\label{pr:IbetaeqItildebetaRiesz}
If $\Omega$ is convex, then for all $\Phi \in X_{\text{Riesz}}$,
\[
 I_{\beta}(\Phi) = \tilde{I}_{\beta}(\Phi),
\]
where
\[
    \tilde{I}_{\beta}(\phiR):=
    \begin{cases}
        0 &: \text{if there exists $\Psi_k \in C_c^\infty(\R^n)$, $\Psi_k\to \Phi$ in $X_{\text{Riesz}}$ }\\
        & \text{such that $\sup_{k}\| \Psi_k\|_{L^\infty(\R^n)} \leq \beta$},\\
        +\infty &: \mbox{otherwise.}
    \end{cases}
\]
\end{proposition} 
\begin{proof}
We first observe that
\begin{equation}\label{eq:prop:1}
 \tilde{I}_{\beta}(\Phi) = 0 \Rightarrow I_{\beta}(\Phi) = 0.
\end{equation}
Indeed, if $\tilde{I}_{\beta}(\Phi) = 0$ then there exists a sequence $\Psi_k \in C_c^\infty(\R^n)$ with 
$\|\Psi_k\|_{L^\infty(\R^n)} \leq \beta$, 
such that $\Psi_k \to \Phi$ in $X_{\text{Riesz}}$. In particular, we have $\|\Psi_k - \Phi\|_{L^q(\R^n)} \xrightarrow{k \to \infty} 0$. 
Then there exists a  subsequence, 
still denoted by $\Psi_k$, such that $\Psi_k$ converges a.e. to $\Phi$,
which implies  that $|\Phi(x)| \leq \beta$  a.e. in $\mathbb{R}^n$, i.e. $\|\Phi\|_{L^\infty(\R^n)} \leq \beta$. 
By the definition of $I_{\beta}$, we have $I_{\beta}(\Phi) = 0$, which proves \eqref{eq:prop:1}.

From \eqref{eq:prop:1} we conclude
$
  \tilde{I}_{\beta}(\Phi) \geq I_{\beta}(\Phi).
$
It remains to prove
$
  \tilde{I}_{\beta}(\Phi) \leq I_{\beta}(\Phi).
$
If the right-hand side is $+\infty$ then there is nothing to show.
Thus, we only need to show 
\begin{equation}\label{eq:prop:goal}
 I_{\beta}(\Phi) = 0 \quad \Rightarrow \quad \tilde{I}_{\beta}(\Phi) = 0 .
\end{equation}
Suppose that $\Phi \in X_{\text{Riesz}}$ and $I_{\beta}(\Phi) = 0$. In order to establish \eqref{eq:prop:goal}, we need to show
\begin{equation}\label{eq:prop:goal2}
 \forall \eps > 0\ \exists \Theta \in C_c^\infty(\R^n;\R^n), \quad \|\Theta\|_{L^\infty(\R^n)} \leq \beta, \quad \|\Theta - \Phi\|_{X_{\text{Riesz}}} < \eps.
\end{equation}
We proceed in several steps.

\underline{Step 1:}
We first show that
\begin{equation}\label{eq:prop:goalstep1}
 \forall \eps > 0\ \exists \Theta_1 \in X_{\text{Riesz}},\ \supp \Theta_1 \subset \subset \R^n, \quad \|\Theta_1\|_{L^\infty(\R^n)} \leq \beta, \quad \|\Theta_1 - \Phi\|_{X_{\text{Riesz}}} < \eps.
\end{equation}
For $m\in\mathbb{N}$, we choose a smooth cut-off function $0\leq \zeta_m\leq 1$,
 such that  $\zeta_m =1$ when $\abs{x}< m$; $\zeta_m =0$ when $\abs{x}> 2m$; and
$\abs{\nabla \zeta_m }\aleq \frac{1}{m}$.
For sufficiently large $m$, we set
\[
\Theta_1 := \zeta_m \Phi . 
\]
It is clear that $\|\Theta_1\|_{L^\infty(\R^n)} \leq \|\zeta_m \Phi\|_{L^\infty(\R^n)} \leq \|\Phi\|_{L^\infty(\R^n)} \leq \beta$.
Thus, we only need to show \[\norm{ \Phi\zeta_m - \Phi }_{L^{q}(\mathbb{R}^n)}+\norm{ {\rm Div}_{\alpha} (\Phi\zeta_m)  - {\rm Div}_{\alpha}\Phi }_{L^{q}(\Omega)} \to 0 . \]
Since $\abs{\Phi\zeta_m} \leq \abs{\Phi}$ and $\zeta_m\to 1$ almost everywhere on $\mathbb{R}^n$,
using the Lebesgue dominated convergence theorem, we have
$\norm{ \Phi\zeta_m - \Phi }_{L^{q}(\mathbb{R}^n)}\to 0$.
Since $m$ is sufficiently large, we may assume without loss of generality that $\zeta_m(x) \equiv 1$ when ${\rm dist} (x, \Omega)\leq 1$.
Since $\Phi \in X_{\text{Riesz}}$,
there exists a sequence $\Phi_k \in C_c^\infty(\R^n)$,
 such that $\|\Phi_k - \Phi\|_{X_{\text{Riesz}}} \xrightarrow{k \to \infty} 0$ and
\[
 \int_{\R^n} \Phi_k \cdot D^\alpha \varphi = -\int_{\Omega} \Div_\alpha (\Phi_k)\, \varphi , \quad \forall \varphi \in C_c^\infty(\Omega).
\]
By the definition of $X_{\text{Riesz}}$-convergence, 
we can take the limits of both sides, which implies
\[
 \int_{\R^n} \Phi \cdot D^\alpha \varphi = -\int_{\Omega} \Div_\alpha (\Phi)\, \varphi , \quad \forall \varphi \in C_c^\infty(\Omega).
\]
Now we claim that $\Div_\alpha (\zeta_m \Phi) \in L^q(\Omega)$. Indeed, let $\varphi \in C_c^\infty(\Omega)$, then 
\[
\begin{split}
\int_{\R^n} \Phi \zeta_m \cdot D^\alpha \varphi =& \int_{\R^n} \Phi \cdot D^\alpha (\zeta_m \varphi) + \int_{\R^n} \Phi \cdot \brac{\zeta_m D^\alpha (\varphi)-D^\alpha (\zeta_m \varphi)}\\
=& \int_{\R^n} \Phi \cdot D^\alpha \varphi + \int_{\R^n} \Phi \cdot \brac{\zeta_m D^\alpha (\varphi)-D^\alpha (\zeta_m \varphi)} .
\end{split}
\]
In the last step we used that $\zeta_m \varphi = \varphi$ by the definition of $\zeta_m$ and the support of $\varphi$.
Using e.g. the Coifman-McIntosh-Meyer commutator estimate (e.g., see \cite[Theorem 6.1]{LS} or \cite[Theorem 3.2.1]{IngmannsMaster}), we have
\[
 \|\zeta_m D^\alpha (\varphi)-D^\alpha (\zeta_m \varphi)\|_{L^{q'}(\R^n)} \aleq [\zeta_m]_{\lip} \|\lapms{1-\alpha} \varphi\|_{L^{q'}(\R^n)},
\]
where $\lapms{1-\alpha}$ denotes the Riesz potential and $q' = \frac{q}{q-1}$. Since $\varphi$ has compact support in the bounded set $\Omega$, we have by Sobolev-Poincar\'e inequality
\[
 \|\lapms{1-\alpha} \varphi\|_{L^{q'}(\R^n)} \aleq_{\Omega} \| \varphi\|_{L^{q'}(\Omega)},
\]
which follows from the usual blow-up argument used for the classical Poincar\'e inequality.
That is, we have shown that for any $\varphi \in C_c^\infty(\Omega)$,
\[
\abs{\int_{\R^n} \Div_\alpha \brac{\Phi \zeta_m -\Phi}\,  \varphi } \equiv
 \abs{\int_{\R^n} \brac{\Phi \zeta_m -\Phi}\cdot D^\alpha \varphi } \leq C(\Omega,q) \|\Phi\|_{L^q(\R^n)}\, \|\varphi\|_{L^{q'}(\R^n)} [\zeta_m]_{\lip}.
\]
Observe that $[\zeta_m]_{\lip}  \aleq \frac{1}{m}$, so we have shown by duality that
\[
 \|\Div_\alpha \brac{\Phi \zeta_m -\Phi}\|_{L^q(\Omega)} \aleq \frac{1}{m}\|\Phi\|_{L^q(\R^n)} \xrightarrow{m \to \infty} 0,
\]
which establishes \eqref{eq:prop:goalstep1}

\underline{Step 2:}
By translation we may assume $\Omega$ is convex and $0 \in \Omega$. For $\rho > 1$,
we set
\begin{equation} \label{rhoOmega,def}
 \Omega_\rho := \rho \Omega = \{\rho x: \quad x \in \Omega\}.
\end{equation}
Then, from \Cref{la:starshapeddist} and \Cref{la:starshapeddist2}, we have $\Omega \subset\subset \Omega_{\rho}$ for $\rho > 1$.

In this step, we  show that
\begin{equation}\label{eq:prop:goalstep2}
\begin{aligned} 
&
\forall \eps > 0\ \exists \Theta_2 \in X_{\text{Riesz}},\ \rho > 1, \supp \Theta_2 \subset \subset \R^n, \Div_\alpha \Theta_2 \in L^q(\Omega_\rho), 
\\
& \|\Theta_2\|_{L^\infty(\R^n)} \leq \beta, \quad \|\Theta_2 - \Phi\|_{X_{\text{Riesz}}} < \eps.
\end{aligned}
\end{equation}
By the results in  Step 1, we only need to show \eqref{eq:prop:goalstep2}  
with $\Phi$ replaced by $\Theta_1$. We let $\Psi:= \Theta_1$ for convenience.
For $\rho > 1$,
\[
 \Psi_\rho(x) := \Psi(x/\rho).
\]
Then we have
\[
 \|\Psi_\rho\|_{L^\infty(\R^n)} = \|\Psi\|_{L^\infty(\R^n)} \leq \beta.
\]
Moreover, in view of \Cref{la:Lpconv}, we have
\[
 \|\Psi_\rho-\Psi\|_{L^q(\R^n)} \xrightarrow{\rho \to 1^+} 0.
\]
It remains to show  $\Div_\alpha \Psi_\rho \in L^q(\Omega_\rho)$ and
\[
 \|\Div_\alpha \Psi_\rho-\Div_\alpha \Psi\|_{L^q(\Omega)} \xrightarrow{\rho \to 1^+} 0.
\]
We first observe that
\[
 \int_{\R^n} \Psi_\rho \cdot D^\alpha \varphi\, dx = \rho^{n-\alpha} \int_{\R^n} \Psi \cdot  D^\alpha \brac{ \varphi\brac{ \rho \ \cdot} } dx.
\]
Thus, from $\varphi(\cdot) \in C_c^\infty(\Omega_\rho)$ 
we have $\varphi\brac{ \rho  \ \cdot} \in C_c^\infty(\Omega)$. 
Then we conclude
\[
 \Div_\alpha \Psi_\rho(x) = \rho^{-\alpha} (\Div_\alpha \Psi)(x/\rho) \quad \text{for a.e. $x \in \Omega_\rho$}.
\]
In particular $\Div_\alpha \Psi_\rho \in L^q(\Omega_\rho)$.
We now have
\[
\begin{split}
 & \|\Div_\alpha \Psi_\rho-\Div_\alpha \Psi\|_{L^q(\Omega)}\\
 &
 \leq \|\Div_\alpha \Psi_\rho-\rho^{-\alpha}\Div_\alpha \Psi\|_{L^q(\Omega)} +  \|\rho^{-\alpha}\Div_\alpha \Psi-\Div_\alpha \Psi\|_{L^q(\Omega)}\\
& = \rho^{-\alpha}\|(\Div_\alpha \Psi)(\cdot/\rho)-(\Div_\alpha \Psi)(\cdot)\|_{L^q(\Omega)} +  (1-\rho^{-\alpha})\|\Div_\alpha \Psi\|_{L^q(\Omega)}\\
& \xrightarrow{\rho \to 1^+} 0,
\end{split}
 \]
where we have used \Cref{la:Lpconv} for the first term and $\Div_\alpha \Psi \in L^q(\Omega)$ for the second term.
This implies that \eqref{eq:prop:goalstep2} is satisfied, by considering $\Psi_\rho$ for $\rho > 1$ close enough to $1$.

\underline{Step 3: Conclusion}
Given $\eps > 0$, we  take $\Psi := \Theta_2$ and pick $\rho > 1$ such that \eqref{eq:prop:goalstep2} is satisfied for $\frac{\eps}{2}$ instead of $\eps$.
Since $\Omega$ is convex, by \Cref{la:starshapeddist} and \Cref{la:starshapeddist2}, 
there exists $D > 0$ such that $\dist(\Omega,\partial \Omega_\rho) > D$. 
We let $\delta_0 := \frac{D}{100}$ and choose $\delta \in (0,\delta_0)$.
Let $\eta \in C_c^\infty(B(0,1))$ 
be the usual symmetric mollifier kernel, and set
\[
 \Psi_\delta := \eta_\delta \ast \Psi.
\]
Since $\supp \Psi \subset \subset \R^n$, we have $\Psi_\delta \in C_c^\infty(\R^n)$ and
\[
 \|\Psi_\delta\|_{L^\infty(\R^n)} \leq \|\Psi\|_{L^\infty(\R^n)} \leq \beta
\]
We also have by usual mollification
\[
 \|\Psi_\delta - \Psi\|_{L^q(\R^n)} \xrightarrow{\delta \to 0} 0.
\]
Lastly, for $\varphi \in C_c^\infty(\Omega)$ we have by Fubini's theorem
\[
 \int_{\R^n} \Psi_\delta \cdot D^\alpha \varphi = \int_{\R^n} \Psi \cdot D^\alpha (\varphi\ast \eta_\delta).
\]
Observe that $\varphi \in C_c^\infty(\Omega)$ implies that $\varphi \ast \eta_\delta \in C_c^\infty(B(\Omega,\delta)) \subset C_c^\infty(\Omega_\rho)$.
Thus, we have
\[
 \int_{\R^n} \Psi_\delta \cdot D^\alpha \varphi = \int_{\R^n} \Div_\alpha \Psi (\varphi\ast \eta_\delta) \quad \forall \varphi \in C_c^\infty(\Omega).
\]
Since $\Div_\alpha \Psi \in L^q(\Omega_\rho)$, we conclude that
\[
 \Div_\alpha \Psi_\delta = (\Div_\alpha \Psi) \ast \eta_\delta \quad \text{in $\Omega$}.
\]
Then, since $\Div_\alpha \Psi \in L^q(\Omega_\rho)$ we have that $(\Div_\alpha \Psi) \ast \eta_\delta$ converges to $\Div_\alpha \Psi$ in $L^q(\Omega)$ as $\delta \to 0$, i.e.
\[
 \|\Div_\alpha \Psi_\delta - \Div_\alpha \Psi\|_{L^q(\Omega)} \xrightarrow{\delta \to 0^+} 0.
\]
Thus, we conclude that
\[
 \|\Psi_\delta - \Psi\|_{X_{\text{Riesz}}} \xrightarrow{\delta \to 0^+} 0.
\]
Now by choosing $\delta>0$ sufficiently small, we have
\[
 \|\Psi_\delta - \Phi\|_{X_{\text{Riesz}}} \leq \|\Psi_\delta - \Psi\|_{X_{\text{Riesz}}} + \|\Psi - \Phi\|_{X_{\text{Riesz}}} \leq \frac{\eps}{2} + \frac{\eps}{2} = \eps.
\]
Letting $\Theta := \Psi_\delta$, we have shown \eqref{eq:prop:goal2},
which implies \eqref{eq:prop:goal}.
Therefore, we have proved $  \tilde{I}_{\beta}(\Phi) \leq I_{\beta}(\Phi)$,
which completes the proof.
\end{proof}

With the help of \Cref{pr:IbetaeqItildebetaRiesz},
 we can now continue with the optimizing problem.
\begin{proof}[Proof of \Cref{lemma:F*andG*}] For $G^*$ the procedure is standard, cf. \cite[Ch. I]{Ekeland99}, and follows from (\ref{def:FenchelC}),
\begin{align*}
  G^*&: Y^* \rightarrow \overline{\R},\\
  G^*(u)&= \sup_{v\in L^q(\Omega)} \left\{\int_{\Omega} v udx - G(v) \right\}
  = \frac{1}{p} \|u+u_N\|_{L^p(\Omega)}^p.
\end{align*}
As for  $F^*$, we follow  \cite[Section 2]{Hintermuller}, with the crucial adaptation of using \Cref{pr:IbetaeqItildebetaRiesz} in the last step
\begin{align*}
  F^*&: V^* \rightarrow  \overline{\R},\\
 F^*(\Psi^*) &= \sup_{\phiR\in V}  \left\{\langle \Psi^*,\phiR\rangle_{V^*,V} - F(\phiR) \right\}
    = \sup_{\phiR\in V}  \left\{\langle \Psi^*,\phiR\rangle_{X^*,X} - I_{\beta}(\phiR) \right\}\\
    &= \sup_{\phiR\in V}  \left\{\langle \Psi^*,\phiR\rangle_{X^*,X} - \tilde{I}_{\beta}(\phiR) \right\}
    = \sup_{\substack{\phiR\in V\cap C_c^\infty(\R^n)\\  \| \phiR\|_{{L^\infty(\R^n)}} \leq \beta}} \langle \Psi^*,\phiR\rangle_{V^*,V}.
\end{align*}
The condition in the last line that we can assume $\Phi \in  C_c^\infty(\R^n)$ is the crucial point of  \Cref{pr:IbetaeqItildebetaRiesz}, and the only place where convexity of $\Omega$ appears.
Finally, by definition we have   $\Lambda^*: Y^*  \rightarrow V^*$. Therefore,
\begin{align*}
 F^*(\Lambda^*u)&= \sup_{\substack{\phiR\in V \cap C_c^\infty(\R^n)\\ \| \phiR\|_{{L^\infty(\R^n)}} \leq \beta}} \langle \Lambda^*u,\phiR\rangle_{V^*,V}
                 = \sup_{\substack{\phiR\in V\cap C_c^\infty(\R^n)\\  \| \phiR\|_{{L^\infty(\R^n)}} \leq \beta}} \langle u,\Lambda \phiR\rangle_{Y^*,Y}\\
                 &= \beta \sup_{\substack{\phiR\in \Cc^\infty(\R^n,\R^n)\\  \| \phiR\|_{{L^\infty(\R^n)}} \leq 1 }} \int_{\R^n} \chi_\Omega u(-\divR \phiR)dx
                 = \beta \VR(\chi_\Omega u;\R^n),
\end{align*}
which concludes the proof. 
\end{proof} 
{From Theorem \ref{theorem:FenchelTheorem}, we have the following result.}
\begin{corollary} \label{cor:dualityR} If $\Omega$ is convex, the problems  (\ref{def:primalProblem1}) and (\ref{def:predual}) are related by
\begin{align*}
\qquad &\min_{\phiR \in  X_{\text{Riesz}}}
\left\{ \frac{1}{q} \|-\divR \phiR\|^{q}_{L^q(\Omega)} - \int_{\Omega} u_N (-\divR \phiR)dx
 + I_{\beta}(\phiR) \right\}\\
  &=  -\min_{u\in L^p(\Omega)} \left\{ \frac{\gamma}{p}\|u-{u}_N\|_{L^p(\Omega)}^p + \beta \VR\left(\chi_\Omega u;\R^n \right) \right\}.
\end{align*}
\end{corollary}

It is important to mention that the (predual) problem (\ref{def:predual}) has at least one solution. Moreover, we have the following results for the  optimality conditions, cf. (\ref{def:OptCond}), 
\begin{lemma} Let $\overline{u}$ be the unique solution for (\ref{def:primalProblem1}) and let $\overline{\phiR}$ be  any solution for (\ref{def:predual}).
Then we have 
\begin{align}
\Lambda^*\overline{ u} \in \partial F(\overline{\phiR})&\Leftrightarrow \left \langle  \Lambda^*\overline{u}, \bm v-\overline{\phiR}\right\rangle \leq 0 \quad \forall \bm v\in X_{\text{Riesz}},\label{def:opt_cond1}\\
-\overline{u} \in \partial G(\Lambda \overline{\phiR})&\Leftrightarrow -\overline{u} = - \left| \divR \overline{\phiR}\right|^{q-2}\divR \overline{\phiR}-u_N.  \label{def:opt_cond2}
\end{align}
\end{lemma}
\begin{proof}
It is clear that (\ref{def:opt_cond1}) follows from (\ref{def:subdiff}). On the other hand,
    if $G$ is {\it G\^ateaux dif{}ferentiable} at $u\in Y$, then $\partial G(u)=\{G'(u) \}$, cf. \cite[Prop. I.5.3]{Ekeland99}. 
    In turn, the following property about  the {\it duality map}, it is also well known: 
\begin{align*}
\begin{aligned}
\partial \|\cdot \|_{L^q(\Omega)}^{q}: L^{q}(\Omega) & \rightarrow L^{p}(\Omega)\\
u &\mapsto \{q |u|^{q-2}u\},
\end{aligned}
\end{align*}
which proves (\ref{def:opt_cond2}) and finishes the proof.
\end{proof}
\subsection*{Gagliardo-Type}
Next, we focus on the Gagliardo case. We refer to \cite{NovagaOnue21} 
where they studied a related problem. As in case of Riesz, we begin by 
establishing the existence and uniqueness of solution to (\ref{def:primalProblem2}).
\begin{lemma}\label{lemma:exisUniq-II} 
For $p\in (1,p^\infty)$, the problem (\ref{def:primalProblem2})  has a unique solution  $\overline{u}\in \BVG(\Omega)\cap L^p(\Omega).$  
\end{lemma}
\begin{proof}
 The proof is similar to the Riesz case in Lemma~\ref{lemma:exisUniq-I}, after using 
 Proposition \ref{weak_compGagl}.
\end{proof}
Now, we characterize the minimizers of  (\ref{def:primalProblem2})   using the  predual strategy
{as discussed in the Riesz case}.  
Note that $u$ does not need to be  extended by zero outside $\Omega$. 
{As a result, our approach is largely motivated by \cite[Section 2]{Burger}.} 
We now study the predual problem associated to (\ref{def:primalProblem2}).  
In a similar way as {in the} Riesz case,  we consider the spaces 
\begin{align}
     X_{\text{Gagliardo}} = \overline{\left\{\Phi: \phiG \in \Cc^1(\Omega \times \Omega),\ \phiG( x, y) = -\phiG( y,x)\right\}}^{\|\cdot\|_{X_{\text{Gagliardo}}}}, \label{def:Xms}
\end{align}
where 
\begin{align*}
\|\Phi \|_{X_{\text{Gagliardo}}}:= \sqrt[q]{ \| \phiG\|^q_{L^q(\Ep_{od} \Omega)} + \|\divG \phiG\|^q_{L^q(\Omega)} }\, ,
\end{align*}
which is well defined because of Lemma \ref{Gagliardo.div.Lp}. Observe that we can equivalently assume 
$\Phi \equiv 0$ in $(\R^n \times \R^n) \sm (\Omega \times \Omega)$ and set
\begin{align*}
\|\Phi \|_{X_{\text{Gagliardo}}}:= \sqrt[q]{ \| \phiG\|^q_{L^q(\Ep_{od} \R^n)} + \|\divG \phiG\|^q_{L^q(\R^n)} } \, .
\end{align*}

As in the {Riesz case}, we will use the indicator function $I_\beta$ for some $\beta > 0$. For 
$\Phi\in X_{\text{Gagliardo}}$, we define 
\begin{align*}
    I_{\beta}(\Phi):=
    \begin{cases}
        0 &: \| \Phi \|_{L^\infty(\R^n \times \R^n)} \leq \beta,\\
        +\infty &: \mbox{otherwise.}
    \end{cases}
\end{align*}

As in the Riesz case, our main novelty is that we are able to pass from $I_\beta$ to a new $\tilde{I}_\beta$ 
which has better approximation properties.
\begin{proposition}\label{pr:IbetaeqItildebetaGag}
If $\Omega$ is convex then for all $\Phi \in X_{\text{Gagliardo}}$,
\[
 I_{\beta}(\Phi) = \tilde{I}_{\beta}(\Phi),
\]
where 
\[
    \tilde{I}_{\beta}(\phiR):=
    \begin{cases}
        0 &: \text{if there exists $\Psi_k \in C_c^\infty(\Omega \times \Omega)$,} 
        \\
        & \text{ $\Psi_k\to \Phi$ in $X_{\text{Gagliardo}}$ such that $\sup_{k}\| \Psi_k\|_{L^\infty(\R^n \times \R^n)} \leq \beta$},\\
        +\infty &: \mbox{otherwise.}
    \end{cases}
\]
\end{proposition}
\begin{proof}
We may assume without loss of generality that $\Omega$ is convex and $0 \in \Omega$.
First, we establish that $I_{\beta}(\Phi) \leq \tilde{I}_{\beta}(\Phi)$ for all  $\Phi\in X_{\text{Gagliardo}}$ .
The case $\tilde{I}_{\beta}(\Phi)=\infty$ is trivial. 
Suppose that $\tilde{I}_{\beta}(\Phi)=0$,
then there exist $\Psi_k\in  C_c^\infty(\Omega\times\Omega)$ with
$\Psi_k (x,y) = -\Psi_k (y,x)$ and $\sup_{k}\| \Phi_k \|_{L^\infty(\R^n \times \R^n)} \leq \beta$,
such that $\Psi_k\to \Phi$ in $X_{\text{Gagliardo}}$.
From the $L^{q}(\Ep^{od} \R^n)$-convergence of $\Psi_k$,
 we can find a subsequence, still denoted by $\Psi_k$, such  that
\[
 \frac{|\Psi_k (x,y)-\Phi(x,y)|}{|x-y|^{\frac{n}{ q }+s}} \xrightarrow{k \to \infty} 0 \quad \text{for $\mathcal{L}^{2n}$-a.e. $(x,y) \in \R^{2n}$},
\]
which in particular implies
\[
 |\Psi_k (x,y)-\Phi(x,y)| \xrightarrow{k \to \infty} 0 \quad \text{for $\mathcal{L}^{2n}$-a.e. $(x,y) \in \R^{2n}$}.
\]
Thus, we have
\[
 |\Phi(x,y)| \leq \beta \quad \text{for $\mathcal{L}^{2n}$-a.e. $(x,y) \in \R^{2n}$},
\]
which implies that $I_{\beta}(\Phi) = 0$
 and proves that $I_{\beta}(\Phi) \leq \tilde{I}_{\beta}(\Phi)$ for all  $\Phi\in X_{\text{Gagliardo}}$.

Now we to prove the opposite direction, i.e.
$I_{\beta}(\Phi) \geq \tilde{I}_{\beta}(\Phi)$ for all  $\Phi\in X_{\text{Gagliardo}}$. If $I_{\beta}(\Phi) = \infty$ there is nothing to show, so we actually need to show
\[
 I_{\beta}(\Phi)=0 \quad \Rightarrow \quad \tilde{I}_{\beta}(\Phi)=0.
\]
Assuming that $\Phi \in X_{\text{Gagliardo}}$ satisfies $\|\Phi\|_{L^\infty(\R^n \times \R^n)} \leq \beta$, 
we  prove that
\begin{equation}\label{eq:Gprop:goal2}
 \forall \eps > 0\ \exists \Theta \in C_c^\infty(\Omega\times \Omega), \quad \|\Theta\|_{L^\infty(\R^n \times \R^n)} \leq \beta, \quad \|\Theta - \Phi\|_{X_{\text{Gagliardo}}} < \eps.
\end{equation}

\underline{Step 1}:
In contrast to the Riesz case, we scale the functions inwards for the Gagliardo case,
which ensures that the mollification produces a function still in $C_c^\infty(\Omega)$.
Using again the notation 
\[
 \Omega_\rho := \rho \Omega = \{\rho x: \quad x \in \Omega\}
\]
 in \eqref{rhoOmega,def}  with $\rho<1$.
Since $\Omega$ is convex and $0 \in \Omega$, we have that $\Omega_\rho \subset \subset \Omega$ for any $\rho < 1$.
We prove that
\begin{equation}\label{eq:Gprop:goal2Step1}
 \forall \eps > 0\ \exists \Theta_{1} \in X_{\text{Gagliardo}},\ \exists\rho < 1,
 \ \supp \Theta_{1} \subset \Omega_\rho \times \Omega_\rho, \; \|\Theta_{1} \|_{L^\infty(\R^n \times \R^n)} \leq \beta, \; \|\Theta_{1} - \Phi\|_{X_{\text{Gagliardo}}} < \eps.
\end{equation}
For  $\rho < 1$, we define 
\[
 \Phi_\rho(x,y):= \Phi \prts{ \frac{x}{\rho} , \frac{y}{\rho} }.
\]
Then we have $\supp \Phi_\rho \subset \Omega_\rho \times \Omega_\rho$ and $\|\Phi_\rho\|_{L^\infty(\R^n \times \R^n)} = \|\Phi\|_{L^\infty(\R^n \times \R^n)} \leq \beta$.
So in order to establish \eqref{eq:Gprop:goal2Step1} we need to show
\begin{equation}\label{eq:scaledPhiGagconv}
 \|\Phi_\rho - \Phi\|_{X_{\text{Gagliardo}}} \xrightarrow{\rho \to 0} 0.
\end{equation}
We first observe that 
\[
 \frac{\Phi_\rho(x,y)}{|x-y|^{\frac{n}{q}}} = \rho^{-\frac{n}{q}} \frac{\Phi(x/\rho,y/\rho)}{|x/\rho-y/\rho|^{\frac{n}{q}}}.
\]
So we have
\[
\begin{split}
 \|\Phi_\rho - \Phi\|_{L^q(\Ep_{od} \R^n)} \leq& \abs{\rho^{-\frac{n}{q}} -1} \norm{\frac{\Phi(x/\rho,y/\rho)}{|x/\rho-y/\rho|^{\frac{n}{q}}}}_{L^q(\R^n \times \R^n)} +
 \left \|\frac{\Phi(x/\rho,y/\rho)}{|x/\rho-y/\rho|^{\frac{n}{q}}} - \frac{\Phi(x,y)}{|x-y|^{\frac{n}{q}}}\right \|_{L^q(\R^n \times \R^n)}\\
 =& \abs{\rho^{-\frac{n}{q}} -1} \rho^{\frac{2n}{q}}\, \norm{\frac{\Phi(x,y)}{|x-y|^{\frac{n}{q}}}}_{L^q(\R^n \times \R^n)} +
 \left \|\frac{\Phi(x/\rho,y/\rho)}{|x/\rho-y/\rho|^{\frac{n}{q}}} - \frac{\Phi(x,y)}{|x-y|^{\frac{n}{q}}}\right \|_{L^q(\R^n \times \R^n)}\\
=& \abs{\rho^{-\frac{n}{q}} -1} \rho^{\frac{2n}{q}}\, \|\Phi\|_{L^q(\Ep_{od} \R^n)} +
 \left \|\frac{\Phi(x/\rho,y/\rho)}{|x/\rho-y/\rho|^{\frac{n}{q}}} - \frac{\Phi(x,y)}{|x-y|^{\frac{n}{q}}}\right \|_{L^q(\R^n \times \R^n)}\\
 \xrightarrow{\rho \to 1}&0,
 \end{split}
 \]
where for the first term we use that $\|\Phi\|_{L^q(\Ep_{od} \R^n)}=\|\Phi\|_{L^q(\Ep_{od} \Omega)} < \infty$, for the second term we use \Cref{la:Lpconv} in $\R^n \times \R^n$.
Moreover, a direct computation from \eqref{eq:divalpha} yields
\[
 \div_\alpha \Phi_\rho(x) = \rho^{-\alpha} (\div_\alpha \Phi)(x/\rho) \quad \text{a.e. }x \in \R^n.
\]
So again with \Cref{la:Lpconv} we have
\[
\begin{split}
 \|\div_\alpha \Phi_\rho-\div_\alpha \Phi\|_{L^{q}(\R^n)} \leq& \abs{\rho^{-\alpha}-1} \|(\div_\alpha \Phi)(\cdot /\rho)\|_{L^{q}(\R^n)}+\|(\div_\alpha \Phi)(\cdot/\rho)-\div_\alpha \Phi\|_{L^{q}(\R^n)}\\
=& \abs{\rho^{-\alpha}-1} \rho^{\frac{n}{{q}}} \|(\div_\alpha \Phi)\|_{L^{q}(\R^n)}+\|(\div_\alpha \Phi)(\cdot/\rho)-\div_\alpha \Phi\|_{L^{q}(\R^n)}\\
\xrightarrow{\rho \to 1}& 0 
 \end{split}
\]
where we used crucially that by the support of $\Phi$ in $\Omega \times \Omega$ we have
\[
 \|\div_\alpha \Phi\|_{L^{q}(\R^n)} = \|\div_\alpha \Phi\|_{L^{q}(\Omega)} < \infty.
\]
This establishes \eqref{eq:scaledPhiGagconv} and thus \eqref{eq:Gprop:goal2Step1} is proven.

\underline{Step 2}: Let $\Theta_1$ and $\rho  \in (0,1)$ be from Step 1. Set $\delta_0:=\frac{D}{100}$ where $D := \dist(\Omega_\rho,\partial \Omega)>0$. 
Let $\eta \in C_c^\infty(B(0,1))$ be the usual symmetric mollifier,
i.e. $\eta \geq 0$ and $\int \eta = 1$.
For  $\delta \in (0, \delta_0)$, we define $\eta_{\delta}(x):= \eta(x/{\delta}) / {\delta}^n$.
Using the notation from \eqref{eq:Gagmollifier}, we define
\[
 \Psi_\delta(x,y) := \prts{ \eta_\delta \ast \Theta_1 } (x,y).
 \]
Then $\Psi_\delta \in C_c^\infty(B( \Omega_{\rho} \times \Omega_{\rho} , \delta) ) \subset C_c^\infty(\Omega \times \Omega)$
and $\|\Psi_\delta\|_{L^\infty(\R^n \times \R^n)} \leq \|\Theta_1\|_{L^\infty(\R^n \times \R^n)} \leq \beta$.
Notice that 
\begin{equation}
\frac{ \prts{ \eta_{\delta}\ast \Theta_1 } (x,y) }{\abs{x-y}^{n/q}}
=
\prts{ \eta_{\delta}\ast \Xi } (x,y)
\end{equation}
where $\Xi(x',y') :=  \Theta_1(x',y')/ \abs{x'-y'}^{n/q} $.
By the definition of $\Theta_1\in L^{q}(\Ep_{od} \R^n)$, we have
$ \Xi \in L^{q}(\R^n\times \R^n)$.
Thus, we have
\begin{equation}\label{eq:Tian1}
 \|\Psi_\delta - \Theta_1\|_{L^{q}(\Ep_{od}^1\R^n)} \xrightarrow{\delta \to 0} 0.
\end{equation}
For any $x\in\mathbb{R}^n$, by letting $y'=y-z$, we obtain
\begin{equation}
\begin{aligned}
& 
\prts{ {\rm div}_{\alpha} \Psi_{\delta} } (x)
=
\prts{ {\rm div}_{\alpha} \prts{\eta_{\delta} \ast \Theta_1 } }(x)
 =
 -  \int_{\mathbb{R}^n} \frac { \prts{\eta_{\delta} \ast \Theta_1 }(x,y) - \prts{\eta_{\delta} \ast \Theta_1 }(y,x) }  { \abs{ x-y}^{n+\alpha} } dy
\\
& =
 -  \int_{\mathbb{R}^n} \prts{ \int_{\mathbb{R}^n}
 \eta_{\delta} (z) \prts{ \Theta_1 (x-z,y-z) -  \Theta_1 (y-z,x-z) } dz} 
 \frac { dy }  { \abs{ x-y}^{n+\alpha} } 
\\
& =
  \int_{\mathbb{R}^n} 
 \eta_{\delta} (z) 
 \prts{ -  \int_{\mathbb{R}^n}
 \frac { \Theta_1  (x-z,y') -  \Theta_1 (y',x-z)  }  { \abs{ (x - z) - y'}^{n+\alpha} }  dy' 
 }
 dz
\\
& =
\int_{\mathbb{R}^n} 
 \eta_{\delta} (z) 
  ({\rm div}_{\alpha} \Theta_1 ) (x-z)
 dz
= 
\prts{ \eta_{\delta} \ast  
  ({\rm div}_{\alpha} \Theta_1 ) }(x).
\end{aligned}
\end{equation}
Thus, we have 
\begin{equation}\label{eq:Tian2}
\begin{aligned}
&
\norm{ {\rm div}_{\alpha} \Psi_{\delta} - {\rm div}_{\alpha} \Theta_1 }_{L^{q}(\R^n)} 
=
\norm{ \eta_{\delta} \ast ( {\rm div}_{\alpha} \Theta_1) - {\rm div}_{\alpha} \Theta_1 }_{L^{q}(\R^n)} 
\to 0
\end{aligned}
\end{equation}
as $\delta\to 0^+$.
Using \eqref{eq:Tian1} and \eqref{eq:Tian2},
for a sufficiently small $\delta$, the function $\Theta:= \Psi_{\delta}$
satisfies the requirements in \eqref{eq:Gprop:goal2}, which completes the proof.
\end{proof}

Now we continue with the optimizing problem.
We  set $V =\left( X_{\text{Gagliardo}}, \|\cdot\|_{ X_{\text{Gagliardo}}}  \right),$ $Y=\left(L^q(\Omega),\|\cdot\|_{L^q(\Omega)} \right)$
 and the operators 
\begin{align}
\begin{aligned}
G: Y  & \rightarrow\overline{\R},\quad    G(v):= \frac{1}{q} \|v\|_{L^q(\Omega)}^{q} - \int_{\Omega} u_N vdx,\\
 F:V & \rightarrow  \overline{\R},\quad
 F(\phiG):= I_{\beta}(\phiG),\\[0.2cm]
  \Lambda: V & \rightarrow Y,\quad
 \Lambda( \Phi) := -\divG  \phiG.
  \end{aligned} \label{def:FandG_G}
\end{align}
Similarly as in Lemma \ref{lemma:F*andG*},  we have  

\begin{corollary} Let $\Omega$ be open, bounded, and convex set, $V =\left( X_{\text{Gagliardo}}, \|\cdot\|_{ X_{\text{Gagliardo}}}  \right),$ $Y=\left(L^q(\Omega),\|\cdot\|_{L^q(\Omega)} \right)$, and let $F,G$ and $\Lambda$ defined as in  (\ref{def:FandG_G}), then  for all $ u \in L^p(\Omega)$
\begin{align*}
 G^*(-u)=& \,\,\frac{1}{p} \| u-u_N\|^p_{L^p(\Omega)}, \quad  \mbox{and }\\
 F^*(\Lambda^* u) =& \,\, \beta \VG(u;\Omega).
\end{align*}
\end{corollary}
This motivates {us to} consider the problem 
\begin{align} \tag{$\mathscr Q_G$} \label{def:predualG}
\inf_{\phiG \in X_{\text{Gagliardo}}} \left \{\frac{1}{q} \|-\divG \phiG\|_{L^q(\Omega)}^q - \int_\Omega u_N (-\divG \phiG)\,dx + I_{\beta}(\phiG)  \right\}.
\end{align}
We have the following result (see, Corollary \ref{cor:dualityR} for the Riesz case).
\begin{corollary} \label{cor:dualityG}
If $\Omega$ is an open, bounded, and convex set, the problems  (\ref{def:primalProblem2}) and (\ref{def:predualG}) are related by
\begin{align*}
&\min_{\phiG \in X_{\text{Gagliardo}}} \left \{\frac{1}{q} \|-\divG \phiG\|_{L^q(\Omega)}^q - \int_\Omega u_N (-\divG \phiG)\,dx + I_{\beta}(\phiG)  \right\}\\
&=  -\min_{u\in L^p(\Omega)} \left\{ \frac{\gamma}{p}\|u-{u}_N\|_{L^p(\Omega)}^p + \beta  \VG (u;\Omega) \right\}.
\end{align*}
\end{corollary}
Finally, we have the following optimality conditions as consequences of \Cref{theorem:FenchelTheorem}.

\begin{corollary} Let $\overline{u}$ be the unique solution to (\ref{def:primalProblem2}) and let $\overline{\Phi}$ be  any solution to (\ref{def:predualG}), then
\begin{align}
\Lambda^*\overline{ u} \in \partial F(\overline{\Phi})&\Leftrightarrow \left \langle  \Lambda^*\overline{u}, \Psi -\overline{\Phi}\right\rangle \leq 0 
\quad \forall \Psi\in X_{\text{Gagliardo}},\label{def:opt_cond1G} \mbox{and }\\
-\overline{u} \in \partial G(\Lambda \overline{\Phi})&\Leftrightarrow -\overline{u} = - \left| \divG \overline{\Phi}\right|^{q-2}\divG \overline{\Phi}-u_N.  \label{def:opt_cond2G}
\end{align}
\end{corollary}

\appendix
\section{Scaling in \texorpdfstring{$L^p$}{Lp}-norms and star-shaped domains}
In this appendix we state and prove for the convenience of the reader some facts about star-shaped domains that are most likely well-known to experts.

Denote the $n-1$-dimensional unit sphere by $\S^{n-1} := \{x \in \R^n: |x| = 1\}$. For $x \in \S^{n-1}$.

\begin{lemma}\label{la:starshapeddist}
Assume $\lambda : \S^{n-1} \to (0,\infty)$ is continuous and consider
\[
 \Omega = \left \{x\in \R^n\sm \{0\}: |x| < \lambda\brac{\frac{x}{|x|}} \right \} \cup \{0\}.
\]
For $\rho > 0$ set 
\[
 \Omega_\rho := \{\rho x: \quad x \in \Omega\}
\]
then we have for any $\rho_1 < \rho_2$
\[
 \dist(\Omega_{\rho_1},\R^n \sm \Omega_{\rho_2}) > 0.
\]
\end{lemma}
\begin{proof}
We first observe 
\begin{equation}\label{eq:23489wiref}
 \partial \Omega = \left \{x\in \R^n \sm \{0\}: |x| = \lambda\brac{\frac{x}{|x|}} \right \}.
\end{equation}

Indeed $\bar{x} \in \partial \Omega$. Since $0 \in \Omega$ and $\Omega$ is open by continuity of $\lambda$, we have $\bar{x} \neq 0$. Then there exists $0 \neq x_k \in \Omega$, $0 \neq y_k \in \R^n \sm \Omega$ such that $\lim_{k} |x_k-\bar{x}| = \lim_{k} |y_k-\bar{x}| =0$. We have 
\[
 |x_k| < \lambda\brac{\frac{x_k}{|x_k|}}, \quad |y_k| \geq \lambda\brac{\frac{y_k}{|y_k|}} \quad \forall k.
\]
Since $x_k, y_k, \bar{x} \neq 0$ these expressions are continuous and passing to the limit as $k \to \infty$,
\[
 |\bar{x}| \leq \lambda\brac{\frac{\bar{x}}{|\bar{x}|}}, \quad |\bar{x}| \geq \lambda\brac{\frac{\bar{x}}{|\bar{x}|}} \quad \forall k.
\]
This implies $|\bar{x}| = \lambda(\frac{\bar{x}}{\bar{x}})$ and thus we have established
\[
 \partial \Omega \subseteq \left \{x\in \R^n \sm \{0\}: |x| = \lambda\brac{\frac{x}{|x|}} \right \}.
\]
Now assume $\bar{x} \in \R^n \sm \{0\}$ with $|\bar{x}| = \lambda\brac{\frac{x}{|x|}}$. Then for $\mu > 0$ we have 
\[
 |\mu\bar{x}| = \mu \lambda\brac{\frac{\mu \bar{x}}{|\mu \bar{x}|}}.
\]
Thus, if $\mu > 1$ we have $\mu \bar{x} \not \in \Omega$ and if $\mu < 1$ we have $\mu \bar{x} \in \Omega$. In particular, 
\[
 x_k := (1-\frac{1}{k}) \bar{x} \in \Omega, \quad y_k := (1+\frac{1}{k}) \bar{x} \not \in \Omega,
\]
and $\lim_{k \to \infty} x_k = \lim_{k \to \infty} y_k = \bar{x}$, so $\bar{x} \in \overline{\Omega} \cap \overline{\R^n \sm \Omega} = \partial \Omega$. This implies
\[
 \partial \Omega \supseteq \left \{x\in \R^n \sm \{0\}: |x| = \lambda\brac{\frac{x}{|x|}} \right \}.
\]
So \eqref{eq:23489wiref} is established.

Next we observe
\[
 \Omega_\rho = \left \{x\in \R^n \sm \{0\}: |x| < \rho \lambda\brac{\frac{x}{|x|}} \right \} \cup \{0\}.
\]
In particular if $\rho_1 < \rho_2$ we have that 
\[
 \Omega_{\rho_1} \cap \brac{\R^n \sm \Omega_{\rho_2} } = \{x: |x| < \rho_1 \lambda\brac{\frac{x}{|x|}} , \quad \text{ and } \quad |x| \geq \rho_2 \lambda\brac{\frac{x}{|x|}}  \}= \emptyset.
\]
Since $\Omega_{\rho_1}$ and $\brac{\R^n \sm \Omega_{\rho_2} }$ are disjoint, and  $\Omega_{\rho_1}$ is bounded we conclude that 
\[
\begin{split}
 \dist(\Omega_{\rho_1},\R^n \sm \Omega_{\rho_2}) =& \dist(\partial \Omega_{\rho_1}, \partial \Omega_{\rho_2})\\
 =& \inf_{x,y \in \R^n} \abs{ \rho_1 \frac{x}{|x|} \lambda(\frac{x}{|x|})-\rho_2 \frac{y}{|y|} \lambda(\frac{y}{|y|})}\\
 =& \inf_{x,y \in \S^{n-1}} \abs{ \rho_1x \lambda(x)-\rho_2 y \lambda(y)}.
 \end{split}
\]
Since $\lambda(\cdot)$ is continuous and $\S^{n-1}$ is compact, this infimum is attained at some $\bar{x}, \bar{y} \in \S^{n-1}$,
\[
\dist(\Omega_{\rho_1},\R^n \sm \Omega_{\rho_2})=\abs{\rho_1 \bar{x} \lambda(\bar{x})- \rho_2 \bar{y} \lambda(\bar{y})}
\]
We claim that $\abs{\rho_1 \bar{x} \lambda(\bar{x})- \rho_2 \bar{y} \lambda(\bar{y})} > 0$. Indeed if this was not the case we would have 
\[
 \rho_1 \bar{x} \lambda(\bar{x}) =  \rho_2 \bar{y} \lambda(\bar{y})
\]
Since the scalar factors $\rho_1, \rho_2, \lambda(\bar{x}),\lambda(\bar{y})$ are all positive -- and $|\bar{x}|=|\bar{y}|=1$ this implies that $\bar{x} = \bar{y}$. Whence we would find 
\[
 \rho_1 \lambda(\bar{x}) = \rho_2 \lambda(\bar{x}),
\]
and thus -- since $\lambda(\bar{x}) \in (0,\infty)$, $\rho_1 = \rho_2$ -- a contradiction to $\rho_1 < \rho_2$. Thus we have established
\[
 \dist(\Omega_{\rho_1},\R^n \sm \Omega_{\rho_2}) > 0.
\]
\end{proof}

In \Cref{la:starshapeddist2}, the continuity of $\lambda$ is not guaranteed for generic star-shaped domain -- even if their boundaries are Lipschitz. We provide two examples in Figure \ref{discontinuous,lambda,example}. The first example is the union of an open disk and an open sector. The second is an open unit disk with a slit, i.e. the ray $\{ (c +1/2, c+1/2) : c\geq 0 \}$ is excluded from the disk. 

\begin{figure} 
\centering
\begin{minipage}{0.3\textwidth}
\centering
\begin{tikzpicture}
\draw[dashed] (70:1) arc (70:-250:1);
\draw[dashed] (70:1) --  (70 : 1.3) arc(70:110:1.3) -- (110: 1.3) -- (110: 1);
\draw[->] (-1.3,0) -- (1.3,0); 
\draw[->] (0,-1.3) -- (0,1.6) ;
\end{tikzpicture}
\end{minipage}
\begin{minipage}{0.3\textwidth}
\centering
\begin{tikzpicture}
\draw[dashed] (0,0) circle (1cm);
\draw[dashed] (45:0.5) --  (45:1);
\draw[->] (-1.3 , 0) -- (1.3,0); 
\draw[->] (0,-1.3) -- (0,1.3) ;
\end{tikzpicture}
\end{minipage}
\caption{Examples of star-shaped sets with discontinuous $\lambda$. Both sets are star-shaped with respect to the origin, and the first has even Lipschitz continuous boundary -- however the conclusions of \Cref{la:starshapeddist} are not true.}
\label{discontinuous,lambda,example}
\end{figure}
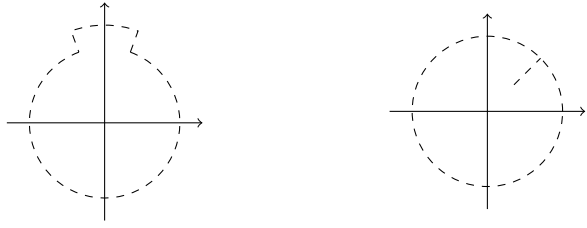

However, the assumptions of the set $\Omega$ in \Cref{la:starshapeddist} are satisfied if $\Omega$ is star-shaped w.r.t to an \emph{open neighborhood} of the origin -- this can be obtained by a careful inspection of the proof below. We will focus on convexity here.
\begin{lemma}\label{la:starshapeddist2}
Let $\Omega$ be an open, bounded, convex set with $0 \in \Omega$, then 
there exists continuous $\lambda : \S^{n-1} \to (0,\infty)$ 
such that
\begin{equation} \label{boundary,of,Omega}
 \Omega = \left \{x\in \R^n\sm \{0\}: |x| < \lambda\brac{\frac{x}{|x|}} \right \} \cup \{0\}.
\end{equation}
In particular the results of \Cref{la:starshapeddist} are true.
\end{lemma}

\begin{proof}
For $x\in\S^{n-1}$, we define
\begin{equation}
 \lambda(x):= \sup \left\{ r\geq 0 : rx \in \Omega  \right\}.
\end{equation}
Since $\Omega$ is open and $0 \in \Omega$ there exists a ball $B(0,a) \subset \Omega$, and thus $\lambda(x) \geq r$ for all $x \in \S^{n-1}$. Since $\Omega$ is bounded there must be some $b > 0$ such that $\lambda(x) \leq b$ for all $x \in \S^{n-1}$.

We first establish \eqref{boundary,of,Omega}. If $x \in \Omega$ then $|x| \frac{x}{|x|} \in \Omega$ and since $0 \in \Omega$ we have that $r \frac{x}{|x|} \in \Omega$ for all $r \in [0,|x|]$. Since $\Omega$ is open, there actually must be some $\delta > 0$ such that $r \frac{x}{|x|} \in \Omega$ for all $r \in [0,|x|+\delta]$. Thus $\lambda(x/|x|) \geq |x|+\delta > |x|$.

On the other hand if $x \in \R^n \sm \{0\}$ and $|x| < \lambda(x/|x|)$, then by definition of $\lambda(\cdot)$ there must be some $r>|x|$ such that $r x/|x| \in \Omega$. Since $0 \in \Omega$ and $\Omega$ is convex we conclude that $x=|x| x/|x| \in \Omega$. Thus \eqref{boundary,of,Omega} is established.

It remains to prove the continuity of $\lambda$ on $\mathbb{S}^{n-1}$. 
Given any $\bar{x} \in \S^{n-1}$,  we let 
$ \{x_k\}_{k=1}^{\infty} \subseteq  \S^{n-1} $ be a sequence
such that $x_k \xrightarrow{k \to \infty} \bar{x}$.

Recall that the open ball $B(0,a) \subset \Omega$. We denote the open cone from $\lambda(\bar{x}) \bar{x}$ to $B(0,a)$ as
\begin{equation}
A:=\{ \theta \lambda(\bar{x})\bar{x} + (1-\theta) z :  z\in B(0,a) , \;  \theta\in [0,1) \}.
\end{equation}
Clearly, $A$ is an open set. 
 Also, whenever $\theta \in [0,1)$ we have that $\theta \lambda(\bar{x})\bar{x} \in \Omega$, by convexity of $\Omega$ and definition of $\lambda(\cdot)$. Since $z$ is taken from an open ball $B(0,a) \subset \Omega$ we conclude that $\theta \bar{x} + (1-\theta) z \in \Omega$. That is we have $A \subset \Omega$.

Similarly, we define  the  open sets $A_k$ by
\begin{equation}
A_k:=\{ \theta \lambda(x_k) x_k + (1-\theta) z :  z\in B(0,a) , \;  \theta\in [0,1) \} \subset \Omega
\end{equation}

Now we assume that there exists $\eps > 0$ and a sequence $x_k \in \S^{n-1}$ converging to $\bar{x} \in \S^{n-1}$ such that $\lambda(x_k) \leq \lambda(\bar{x})-\eps$. Then $x_{k} \lambda(x_k) \subset A$ when $k$ is sufficiently large, see Figure~\ref{discontinuous,lambda,temp1}. 
Thus we have lower semicontinuity of $\lambda$: \[\lambda(\overline{x})\leq \liminf_{\S^{n-1} \ni x \to \bar{x}} \lambda(x).\]

On the other hand, if there exists $\eps > 0$ and a sequence $x_k \in \S^{n-1}$ converging to $\bar{x} \in \S^{n-1}$ such that $\lambda(x_k) \geq \lambda(\bar{x})+\eps$. Then we have that $\bar{x} \lambda(\bar{x}) \in A_{k}$ for all large $k$, see Figure~\ref{discontinuous,lambda,temp1}. Thus we have established upper semicontinuity of $\lambda$
\[
\lambda(\overline{x})\geq \limsup_{\S^{n-1} \ni x \to \bar{x}} \lambda(x).
\]
Therefore, we have proved the continuity of $\lambda$.

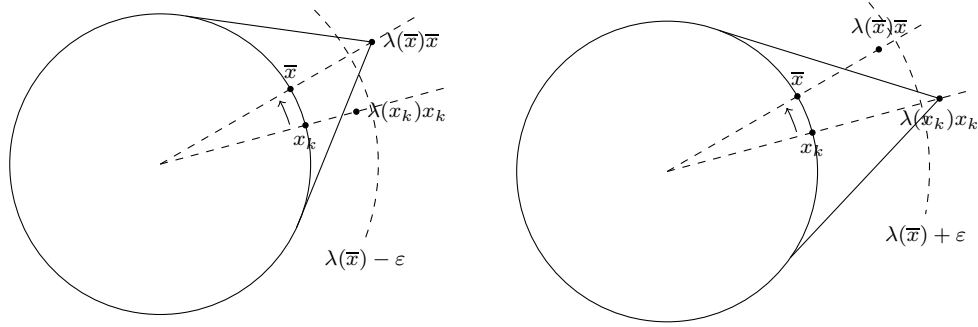
\begin{figure} 
\centering
\begin{minipage}{0.3\linewidth}
\begin{tikzpicture}
\draw[thin] (0,0) circle (2cm);
\filldraw[black] (30:2) circle (1pt) node[anchor=south]{\tiny $\overline{x}$};
\filldraw[black] (15:2) circle (1pt) node[anchor=north]{\tiny $x_k$};
\draw[->,thin] (17:1.8)  arc (17:28:1.8);
\draw[dashed] (45:2.9)  arc (45:-20:2.9) node[anchor=north]{\tiny $\lambda(\overline{x})-\varepsilon$};
\draw[thin] (80:2) -- (30:3.25) --(-25:2);
\filldraw[black] (30:3.25) circle (1pt) node[anchor=west]{\tiny $\lambda(\overline{x})\overline{x}$};
\filldraw[black] (15:2.7) circle (1pt) node[anchor=west]{\tiny $\lambda(x_k)x_k$};
\draw[dashed] (0,0) -- (30:4) ;
\draw[dashed] (0,0) -- (15:4);
\end{tikzpicture}
\end{minipage}
\qquad \qquad
\begin{minipage}{0.3\linewidth}
\begin{tikzpicture}
\draw[thin] (0,0) circle (2cm);
\filldraw[black] (30:2) circle (1pt) node[anchor=south]{\tiny $\overline{x}$};
\filldraw[black] (15:2) circle (1pt) node[anchor=north]{\tiny $x_k$};
\draw[->,thin] (17:1.8)  arc (17:28:1.8);
\draw[dashed] (40:3.5)  arc (40:-10:3.45) node[anchor=north]{\tiny $\lambda(\overline{x})+\varepsilon$};
\draw[thin] (70:2) -- (15:3.75) --(-35:2);
\filldraw[black] (30:3.25) circle (1pt) node[anchor=south]{\tiny $\lambda(\overline{x})\overline{x}$};
\filldraw[black] (15:3.75) circle (1pt) node[anchor=north]{\tiny $\lambda(x_k)x_k$};
\draw[dashed] (0,0) -- (30:4) ;
\draw[dashed] (0,0) -- (15:4.2);
\end{tikzpicture}
\end{minipage}
\caption{\label{discontinuous,lambda,temp1} 
Assuming that the ball $B(0,a)$ in the proof is actually equal to $B(0,1)$ (which can always be obtained by scaling) the above figure explains the proof of \Cref{la:starshapeddist2}. \\
Left: if $\lambda(x_k) < \lambda(\bar{x})-\eps$ and $x_k$ is sufficiently close to $\bar{x}$ then $\lambda(x_k) x_k$ must belong to the cone $A$. Right: if $\lambda(x_k) > \lambda(\bar{x}) + \eps$ and $x_k$ is sufficiently close to $\bar{x}$ then $\bar{x}$ must belong to $A_k$ (using that the cone $A_k$ has a minimal aperture that does not change and is determined by $B(0,a)$ as $k$ changes)}

\end{figure}

\end{proof}

\begin{remark}
We leave the technical details to the reader, but observe that the lower semicontinuity of $\lambda$ holds under the assumption that $\Omega$ is open and star-shaped. It is the upper semiconintuity of $\lambda$ that requires the center of $\Omega$ containing an open neighborhood of the origin $B(0,a)$ (which in particular is a consequence of convexity and openness).
\end{remark}

\begin{lemma}\label{la:Lpconv}
Let $\Omega \subseteq \R^n$ be an open domain star-shaped with respect to the origin. Fix $p \in [1,\infty)$, let $f \in L^p(\Omega)$ and set for $\rho > 1$
\[
 f_\rho := f(\cdot/\rho).
\]
Then
\[
 \|f_\rho -f \|_{L^p(\Omega)} \xrightarrow{ \rho \to 1^+} 0.
\]
\end{lemma}
\begin{proof}
Let $\eps > 0$. Since $p \in [1,\infty)$ we have $C^0(\overline{\Omega})$ is dense in $L^p(\Omega)$, and thus there exists $g \in C_c^0(\R^n)$ with
\[
 \|f-g\|_{L^p(\Omega)} \leq \eps.
\]
Set
\[
 g_\rho := g(\cdot/\rho).
\]
Since $\Omega$ is star-shaped with respect to the origin,
\[
 \|f_\rho -g_\rho\|_{L^p(\Omega)} = \rho^{-\frac{n}{p}} \|f -g\|_{L^p(\frac{1}{\rho}\Omega)} \overset{\rho > 1}{\leq}\|f -g\|_{L^p(\Omega)} \leq \eps.
\]
Then we have  for any $\rho > 1$,
\begin{equation}\label{eq:frhomf12313}
 \|f_\rho -f \|_{L^p(\Omega)} \leq \|f_\rho -g_\rho\|_{L^p(\Omega)} + \|f -g\|_{L^p(\Omega)} + \|g-g_\rho\|_{L^p(\Omega)} \leq 2\eps + \|g-g_\rho\|_{L^p(\Omega)}.
\end{equation}
Take $R > 0$ such that
\[
 \supp g \subset B(0,R/4).
\]
Since $g$ has compact support we can find such an $R$. Then $g$ is uniformly continuous on $\overline{B(0,R)}$ and thus there exists some $\rho_0 \in (1,2)$ such that
\[
 |g(x) - g(x/\rho)|\leq \frac{\eps}{|B(0,R)|^{\frac{1}{p}}} \quad \forall x \in \overline{B(0,R)}, \forall \rho \in [1,\rho_0].
\]
On the other hand if $x \not \in B(0,R)$ then
\[
 g(x)=g(x/\rho) = 0 \quad \forall \rho \in [1,2].
\]
Thus we have
\[
 \|g - g_\rho\|_{L^\infty(\R^n)} < \frac{\eps}{|B(0,R)|^{\frac{1}{p}}} \quad \forall \rho \in [1,\rho_0]
\]
and thus
\[
 \|g-g_\rho\|_{L^p(\Omega)} =\|g-g_\rho\|_{L^p(B(0,R))} \leq |B(0,R)|^{\frac{1}{p}}\, \|g - g_\rho\|_{L^\infty(\R^n)} \leq \eps.
\]
Combining this with \eqref{eq:frhomf12313}, we have shown
\[
 \|f_\rho -f \|_{L^p(\Omega)} \leq 3\eps,
\]
which holds for any $\rho \in [1,\rho_0)$. We can conclude.
\end{proof}

\bibliographystyle{abbrv}
\bibliography{refs}

\end{document}